\documentclass[12pt,a4paper,reqno]{amsart}
\usepackage{fix-cm}
\usepackage[utf8]{inputenc}
\usepackage{geometry}
\usepackage{amsmath, amsthm, cases}
\usepackage{amsfonts}
\usepackage{amssymb}
\usepackage{amsthm}
\usepackage{mathrsfs}
\usepackage{color}
\usepackage{float}
\usepackage{epsfig}
\usepackage{graphicx}
\usepackage{caption}
\usepackage{subcaption}
\usepackage{hyperref}
\usepackage[noabbrev]{cleveref}
\numberwithin{equation}{section}

\newcommand{\la}{\langle}
\newcommand{\ra}{\rangle}

\newcommand{\pa}{\partial}

\newcommand{\R}{\mathbb{R}}
\newcommand{\s}{\mathbb{S}}

\newcommand{\x}{\boldsymbol{x}}

\newtheorem{theorem}{Theorem}[section]

\newtheorem{lemma}[theorem]{Lemma}
\newtheorem{definition}[theorem]{Definition}
\newtheorem{remark}[theorem]{Remark}
\newtheorem{proposition}[theorem]{Proposition}
\newcommand{\Ito}{It\^o\ }
\newcommand{\per}{\mathrm{per}}
\newcommand{\M}{{M_{\eta,\phi_2(t)}}}
\newcommand{\tf}{\tilde f}
\newcommand{\Ind}[1]{\mathbf{1}_{\left\{#1\right\}}}

\begin{document}
	
\title[phase transition of the kinetic J-K models]{Phase transition of the kinetic Justh-Krishnaprasad type model for nematic alignment}
\author{Seung-Yeal Ha$^1$, Hui Yu$^2$, Baige Zhou$^3$}
\address{$^1$Department of Mathematical Sciences and Research Institute of Mathematics, Seoul National University, Seoul, 08826, Republic of Korea.}
\email{syha@snu.ac.kr}
\address{$^2$ School of Mathematics and Computational Science, Xiangtan University, Xiangtan, 411105, China.}
\email{huiyu@xtu.edu.cn}
\address{$^3$ Department of Mathematical Sciences, Tsinghua University, Beijing, 100084, China.}
\email{zbg22@mails.tsinghua.edu.cn}	
 
\begin{abstract}
We present a stochastic Justh-Krishnaprasad flocking model and study the phase transition of the Vlasov-McKean-Fokker-Planck (VMFP) equation, which can be obtained in the mean-field limit. 
To describe the alignment, we use order parameters in terms of the distribution function of the kinetic model. 
For the constant noise case, we study the well-posedness of the VMFP equation on the torus. 
Based on regularity, we show that the phenomenon of phase transition is only related to the ratio between the strengths of noise and coupling. 
In particular, for the low-noise case, we derive an exponential convergence to the von-Mises type equilibrium, which shows a strong evidence for the nematic alignment. 
The multiplicative noise is also studied to obtain a non-symmetric equilibrium with two different peaks on the torus.    
\end{abstract}	
	
\maketitle	
\section{Introduction} 
The emergence of collective behaviors in complex systems is ubiquitous in our nature, for instance, flocking of birds \cite{chate2008modeling,toner1998flocks,MR3363421}, swarming of fish \cite{degond2008large,MR2111591}, and aggregation of bacteria \cite{caratozzolo2008self,copeland2009bacterial,degond2015continuum,degond2018age}, etc. 
These collective dynamics has been a hot topic in physics \cite{berger1984microscopic,mehandia2008collective,sokolov2007concentration}, mathematics \cite{bellomo2022mathematical,carrillo2010asymptotic,coleman2017mathematics}, control theory \cite{MR3853918,MR1986266,olfati2007consensus,MR2531514} and biological systems \cite{del1985quantum,deutsch2012collective,MR1924782,MR2111591}. Collective patterns are the results of communications between agents. Therefore, to understand biological mechanisms behind the collective behaviors of agents, one needs to take into account
of communications between constituent agents: how the particles sense their neighbors and how they react to their
neighbors’ stimulus. For the modeling of such collective motions, several phenomenological models have been studied in literature. To name a few, the Kuramoto model \cite{kuramoto1975international}, the Vicsek model \cite{MR3363421}, and the Cucker-Smale model \cite{cucker2007emergent}, etc. Among them, the modeling of nematic alignment has also received a lot of attentions. Nematic alignment was first observed in a group of Myxobacteria \cite{degond2015continuum,degond2018age} in which bacteria move either in the
same direction or in the opposite direction and they form a traveling band asymptotically. To investigate the temporal evolutions in bacteria, a specialized model was introduced in \cite{ha2020stochastic,ha2022emergent}, which is called the generalized Justh-Krishnaprasad (J-K) model. Based on the original model of Justh and Krishnaprasad \cite{justh2003steering,justh2002simple}, where all particles move with a unit speed on a planar domain, a nontrival communication function is adopted to fit the double-heading angle case. Using the proposed model, the authors in \cite{ha2020stochastic,ha2022emergent} show nematic alignment at the particle level for both constant noise case and multiplicative noise case under suitable conditions by means of order parameters. Nematic alignment has also been studied in other models: for example, in \cite{cho2016emergence2,cho2016emergence} a bi-cluster phenomenon is achieved by the coupling in the particle model; in physics literature \cite{perepelitsa2022mean,sese2021phase}, agent-based models are used to descirbe nematic alignment, and at the same time, the numbers of the particles in both directions are the same in equilibrium state.

We should also emphasize that current research pays significant attention to multi-scale models that connect agent-based approaches to continuum
approaches \cite{chate2008modeling,fetecau2011collective,ponisch2017multiscale}. Instead of tracking the position and the velocity of each agent, the kinetic model deals with the temporal-phase space evolution of one-particle distribution function. At the mesoscopic level, the Boltzmann-type models can exhibit a variety of spatio-temporal patterns \cite{carrillo2014non,eftimie2007complex}. The mean-field kinetic models can be derived formally using the BBGKY hierarchy \cite{king1975bbgky} from particle models using the McKean-Vlasov process, under sufficient conditions with respect to the interaction kernels. This topic is also called propagation of chaos, which is studied on the course of Sznitman at Saint-Flour \cite{sznitman1991topics} in 1991. Recently, the review of Jabin and Wang \cite{MR3858403} focuses on McKean mean-field systems with singular kernels by using the relative entropy. In the macroscopic scale, hydrodynamic models can be derived from the kinetic models by taking velocity moments together with suitable closure conditions \cite{degond2020nematic,poyato2017euler}. Compared to particle models, hydrodynamic models are more efficient for large number of agents, because they simply encode different particle quantities into simple averages, such as density, momentum and energy. 

In the study of equilibria and convergence, phase transition plays a crucial role in the non-equilibrium process, e.g. change of states in liquid crystal \cite{brazovskii1975phase} as the temperature varies. Phase transition appears in several contexts, such as the emergence of flocking dynamics near the critical mass of self-propelled particles \cite{barbaro2016phase,carrillo2020long,MR3067586,MR3305654,frouvelle2012dynamics,MR4344428} and interacting particle systems of continuous Markov process in probability \cite{dotsenko1983critical,liggett1985interacting}. These
models exhibit a kind of competitions between alignment and noise. For the well-known Vicsek model, phase transition is discussed in \cite{MR3305654,frouvelle2012dynamics} for a spatially homogeneous case, where they also obtain the critical point between drift and diffusion. On one hand, below this value, the only equilibrium is isotropic with respect to velocity direction and is stable unconditionally. On the other hand, when the noise strength is above the threshold, a bifurcation of von-Mises distribution with a constant mean orientation emerges, and it is unstable. They also show that there exists an exponential convergence to the equilibrium,  if initial datum is close to an invariant measure. This kind of questions have also been studied in the Cucker-Smale type model using asymptotic analysis \cite{barbaro2016phase}.

To fix the idea, we begin with the particle J-K model. Let $\x_t^j\in \R^2$ and $\theta_t^j\in \s^1$ be the position and heading angle processes of the $j$-th particle at time $t$, and we assume that their dynamics are governed by the following system of coupled SDEs:  
	\begin{equation}\label{original particle system}
		\left\{\begin{aligned}
			d\x_t^j&=\left(\cos\theta_t^j, \sin\theta_t^j\right)dt,\quad t>0,\quad j\in[N]:=\{1,\ldots,N\},\\
			{d\theta_t^j}&=\frac{\kappa}{N}\sum_{k=1}^{N}\psi\left(\vert \x_t^k-\x_t^j\vert\right)\sin\left(2\left(\theta_t^k-\theta_t^j\right)\right)dt+\sqrt{2\sigma}dB_t^j,
		\end{aligned}\right.
	\end{equation}
	where $\kappa$ is a positive coupling constant, $\vert\cdot\vert$ denotes the standard $\ell^2$-norm in $\mathbb{R}^2$ and $\psi$ is a Lipschitz continuous communication weight function satisfying the boundedness, positivity and regularity conditions:
	\begin{equation}\label{A1}
	    0<\psi(r)\leq\psi(0)\ \quad\forall \ r\geq 0; \quad
     [\psi]_{\rm{Lip}}<\infty.
	\end{equation}
  In the heading angle dynamics, $\sigma>0$ is the strength of the noise, and $\{B_t^j\}_{1\leq j\leq N}$ is the collection of $N$ i.i.d. one-dimensional Brownian motions.

  Note that if $\psi$ is a constant, then the heading angle dynamics reduces to the stochastic Kuramaoto model \cite{ha2008particle} for identical oscillators. In the previous paper \cite{ha2020stochastic}, the nematic alignment has been obtained for the constant noise case. For the multiplicative noise, there is also related work \cite{ha2022emergent} for an agent-based model.

In this paper, we deal with the kinetic J-K model for both constant noise and multiplicative noise and study their equilibria as well as phase transition. Therefore, we generalize the noise of the particle model in Section $2$ and study its corresponding mean-field model. 

We emphasize that for the mean-field Vicsek model, the equilibrium states are von-Mises distributions \cite{frouvelle2012dynamics}. Similarly, we can follow the procedures in \cite{frouvelle2012dynamics} and obtain corresponding equilibrium for our proposed model. 
 Since the interaction term has the form of convolution, the stationary states are more complicated. 
  In this paper, we use the LaSalle invariance principle to show the convergence to an equilibrium using $L^2(\s)$ estimate and the relative entropy method \cite{carrillo2020long}. 
  Note that we cannot write the interaction term into a gradient form which is needed to be convex in Langevin equation. Thus, instead of using the Logarithmic Sobolev inequalities directly, we analyze the convergence in a more general way. 
  
\vspace{0.2cm}
  The rest of the paper is organized as follows. In Section $2$, we describe the generalized particle J-K model and derive the corresponding mean-field equation. Then we provide a self-contained proof for the well-posedness and regularity with a constant noise in spatially homogeneous case. In Section $3$, we consider constant noise case. In this case, we use a free energy to study the asymptotic behavior of the solution. We also determine all the steady states and show that the value $\sigma/\kappa=1/4$ is a critical value. When the ratio is over the threshold, uniform distribution is the only equilibrium. However, below this critical value, there are two kinds of steady states (uniform distribution and a type of generalized von-Mises distribution). For the subcritical case, we study the convergence to the uniform distribution both in $L^2$ norm and the relative entropy. In contrast, for the supercritical case $\sigma/\kappa<1/4$, we show that there is an exponential convergence to the von-Mises type equilibrium. In Section $4$, we study the nonconstant noise. To obtain an explicit non-symmetric steady state, we consider a special ansatz of noise and study its corresponding kinetic equation. Phase transition is also included in nonconstant noise case. Finally, Section $5$ is devoted to a brief summary of main results and some remaining issues for a future work.

\section{Mean-field limit of the Justh-Krishnaprasad model}	
In this section, we discuss a formal kinetic J-K model which can be obtained from the particle model \eqref{particle system} below in the mean-field limit $(N\to\infty)$, and we study its global well-posedness and regularity. 

\vspace{0.15cm}
Let $\x_t^j\in \R^2$ and $\theta_t^j\in \s^1$ be the position and heading angle of the $j$-th particle at time $t$, whose dynamics is governed by the following system of coupled SDEs:  
	\begin{equation}\label{particle system}
		\left\{\begin{aligned}
			d\x_t^j&=\left(\cos\theta_t^j, \sin\theta_t^j\right)dt, \ \ t>0,\ \ j\in [N],\\
			{d\theta_t^j}&=\frac{\kappa}{N}\sum_{k=1}^{N}\psi\left(\vert \x_t^k-\x_t^j\vert\right)\sin\left(2\left(\theta_t^k-\theta_t^j\right)\right)dt+\sqrt{2\sigma D_t^j(X_t,\Theta_t)}dB_t^j.
		\end{aligned}\right.
	\end{equation}
 Throughout the paper, we denote the position and heading angle configurations by 
 $(X_t,\Theta_t)$:
	\[
	X_t=\left(\x_t^1,\cdots,\x_t^N\right)\in\R^{2N},\qquad\Theta_t=\left(\theta_t^1,\cdots,\theta_t^N\right)\in\s^N.
	\] 
  Here, the new term $D_t^j=D_t^j(X_t,\Theta_t)$ compared to \eqref{original particle system} is a bounded stochastic process with the following general form:
	 \begin{equation}\label{A3}
	 	D_t^j(X_t,\Theta_t)=\left(\frac{1}{N}\sum_{k=1}^{N}\mathfrak{b}(\x_t^k-\x_t^j,\theta_t^k-\theta_t^j)\right)^2,
	 \end{equation}
 	where $\mathfrak{b}=\mathfrak{b}(\x,\theta)$ is a Lipschitz continuous function such that
    \begin{equation}\label{A4}
        0\leq\mathfrak{b}(\x,\theta)\leq 1\quad \forall (\x,\theta)\in\R^2\times\s^1;\quad \mathfrak b(\boldsymbol{0},0)=0.
    \end{equation}  
Since the coefficients on the right-hand side of \eqref{particle system} are Lipschitz continuous and have sublinear growth in state variables, by the classical Cauchy-Lipschitz theory of stochastic differential equations \cite[Chapter 5]{karatzas1991brownian}, system \eqref{particle system} is globally well-posed for $(\x,\theta)\in \R^2\times\s ^1$.

  Note that depending on how to define stochastic integrals, 
  we have three types of Fokker-Planck equations:
  \begin{equation}\label{F-P}
		\left\{\begin{aligned}
			&\pa_t f(t,\x,\theta)+(\cos\theta,\sin\theta)\cdot\nabla_{\x} f+\kappa\pa_{\theta}\big(L[f]f\big)=\sigma\pa_{\theta\theta} \big(D[f]f\big),\\
			&\pa_t f(t,\x,\theta)+(\cos\theta,\sin\theta)\cdot\nabla_{\x} f+\kappa\pa_{\theta}\big(L[f]f\big)=\sigma\pa_{\theta} \Big(\sqrt{D[f]}\pa_{\theta}\big(\sqrt{D[f]}f\big)\Big),\\
            &\pa_t f(t,\x,\theta)+(\cos\theta,\sin\theta)\cdot\nabla_{\x} f+\kappa\pa_{\theta}\big(L[f]f\big)=\sigma\pa_{\theta} \big(D[f]\pa_{\theta}f\big),
		\end{aligned}\right.
	\end{equation}
 	where the functionals $L[f]$ and $D[f]$ are given as follows:
    \begin{equation}\label{def_L_D}
		\left\{\begin{aligned}
			&L[f](t,\x,\theta)=\int_{\R^2\times \s^1}\psi(\vert\x_*-\x\vert)\sin 2(\theta_*-\theta)f(t,\x_*,\theta_*)d\x_*d\theta_*,\\
			&D[f](t,\x,\theta)=\left(\int_{\R^2\times \s^1}\mathfrak{b}(\x_*-\x,\theta_*-\theta)f(t,\x_*,\theta_*)d\x_* d\theta_*\right)^2.
            \end{aligned}\right.
	\end{equation}	
For the Fokker-Planck equations \eqref{F-P}, the \Ito form in the first line in \eqref{F-P} is a good approximation of discrete noise effect, which is widely accepted in mathematical finance \cite{MR2039136}. 
  The Stratonovich integral in the middle line is often used in engineering and physics communities \cite{gardiner2009stochastic}. 
  The equation in the last line, sometimes known as the kinetic form \cite{klimontovich1990ito}, is also used to study the Maxwellian equilibrium states. Three cases coincide with each other, when the noise term is constant $D[f]\equiv D$ (constant). 
  In this paper, we use the kinetic form $\eqref{F-P}_3$, since we want to analyze the equilibrium state and its asymptotic stability for the mean-field model.

\subsection{Mean-field limit in the kinetic form}

Note that for the particle model \eqref{particle system}, when $N\to\infty$, it is quite difficult to analyze the dynamics of the whole system. To overcome this problem, we use the mean-field approximation by the one-particle distribution function $f=f(t,\x,\theta)$, which gives the probability of the particle lying on position $\x$ with the heading angle $\theta$ at time $t$. 
We want to derive the equation that the distribution function satisfies. 
For all the noises that we consider in the particle model \eqref{particle system}, we assume that the conditions in \eqref{A1}, \eqref{A3} and \eqref{A4} hold. 
By the process of the Vlasov-McKean limit, we can get the Fokker-Planck equation for the stochastic particle model:
\begin{equation}\label{IBVP}
    \left\{\begin{aligned}
       &\pa_t f+(\cos\theta,\sin\theta)\cdot\nabla_{\x} f+\kappa\pa_\theta \left(L[f]f\right)=\sigma\pa_\theta\left(D[f]\pa_\theta f\right),\quad \x\in \mathbb{R}^2,\; \theta \in \s^1, \; t >0, \\
&f|_{t=0}=f_0,  \\
&f(t,\x,\theta) = f(t,\x, \theta+2\pi), \qquad
\lim_{|\x| \to \infty}f(t,\x,\theta) = 0,  
    \end{aligned}\right.
\end{equation}
where $L[f]$ and $D[f]$ are the heading angle alignment force and diffusion coefficient given in \eqref{def_L_D}. 

\subsection{The spatially homogeneous J-K model with constant noise}
In this section, we consider the spatially homogeneous case and the constant noise term, i.e., the distribution function $f$ is independent of the position $\x$ and $D[f]\equiv1$. 
In this setting, $f$ satisfies
\begin{equation}\label{IBVP_hom}
    \left\{\begin{aligned}
    &\pa_tf+\kappa\pa_\theta\big(L[f]f\big)=\sigma\pa_{\theta\theta} f, \quad \theta \in \s^1, \; t >0, \\
&f|_{t=0}=f_0,\quad f(t,\theta) = f(t,\theta+2\pi),
&
    \end{aligned}\right.
\end{equation}
where $L[f]$ is given as follows:
\[
	L[f](t,\theta)=\int_{\s^1}\sin \left(2(\theta_*-\theta)\right)f(t,\theta_*)d\theta_*.
\]

\subsubsection{Well-posedness and regularity}
Before discussing the global existence of classical solutions, we introduce the function spaces and some associated norms:
\begin{equation*}
    \left\{\begin{aligned}
        &L_{\per}^2:=\left\{\varphi:\s^1\rightarrow\R \Big\vert\int_{\s^1}\varphi^2d\theta<\infty \right\}, \ C_{\per}=C_{\per}^0:=\left\{\varphi:\s^1\to\mathbb{R} \ \text{is} \ \text{continuous}\right\}\\ 
        &H_{\per}^m:=\left\{\varphi:\s^1\rightarrow\R \Big\vert\frac{d\varphi}{d\theta}\in H_{\per}^{m-1} \ \text{in} \ \text{weak} \ \text{sense}\right\},\quad\ m\in \mathbb{Z}_0:=\{0\}\cup\mathbb{N},\\
        &C_{\per}^m:=\left\{\varphi:\s^1\to\mathbb{R}\Big\vert\frac{d\varphi}{d\theta}\in C_{\per}^{m-1}\right\},\quad m\in\mathbb{Z}_0,\\
        &\Vert \varphi\Vert_{L_{\per}^2}:=\left(\int_{\s^1}\varphi^2d\theta\right)^{\frac{1}{2}}, \quad
\Vert \varphi\Vert_m:=\Vert \varphi\Vert_{H_{\per}^m}=\left(\sum_{k=0}^{m}\int_{\s^1}\left\vert\frac{d^k\varphi}{d\theta^k}\right\vert^2d\theta\right)^{\frac{1}{2}}.
    \end{aligned}\right.
\end{equation*}
For a time-dependent function $h=h(t,\theta)$, we use handy notations throughout the paper:
\begin{equation*}
    \Vert h(t)\Vert _{L_{\per}^2}\equiv \Vert h(t,\cdot)\Vert_{L_{\per}^2}\quad \text{and} \quad \Vert h(t)\Vert _{m}\equiv \Vert h(t,\cdot)\Vert_{H_{\per}^m}. 
\end{equation*}
\begin{definition}
    We say that a function $f$ is a weak solution of \eqref{IBVP_hom} with initial probability distribution $f_0\in L_{\per}^2$, if the following relations hold:
    \begin{itemize}
    \item[(i)] $f\in L^2\left(\mathbb{R}_+;H_{\per}^1\right)$, $\pa_t f\in L^2\left(\mathbb{R}_+; H_{\per}^{-1}\right)$;
    \vspace{0.1cm}
     \item[(ii)] $\left\la\la\pa_t f,\varphi\right\ra\ra+\sigma\left(\pa_\theta f,\pa_\theta \varphi\right)=\kappa \left(L[f]f,\pa_\theta \varphi\right)\quad \forall \varphi\in H_{\per}^1, \text{ a.e. } \ t\geq 0$;
     \vspace{0.1cm}
        \item[(iii)] $f(t)\to f_0 \ in \ L_{\per}^2$ as $t\to 0$,
    \end{itemize}
    where $\la\la\cdot,\cdot\ra\ra$ is the dual pairing between $H_{\per}^1$ and $H_{\per}^{-1}$, and $(\cdot,\cdot)$ is the inner product in $L_{\per}^2$. 
    \end{definition}
Note that the drift term in \eqref{IBVP_hom} is nonlinear. Therefore, we cannot use the result for linear parabolic equations directly. In the next theorem, we summarize the well-posedness and regularity of \eqref{IBVP_hom}.   
\begin{theorem}\label{thm:pde}
    Let $f_0\in H_{\per}^m$, $m>\frac{5}{2}$ be the initial probability measure. Then there exists a unique strong solution for \eqref{IBVP_hom}. Moreover, this solution is global-in-time, nonnegative and smooth.
\end{theorem}
\begin{proof}
Since the proof is very lengthy, we briefly outline our proof strategy in several steps:
\begin{itemize}
    \item 
    Step A: We derive the equation for the mean-zero function $\tilde{f}(t,\theta):=f(t,\theta)-\frac{1}{2\pi}$.
    \vspace{0.1cm}
    \item 
    Step B: We obtain a local well-posedness of $\tf$ in $[0,T)$, where a finite final time $T$ depends on $\sigma$, $\kappa$ and $\left\Vert \tf_0\right\Vert_{L_{\per}^2}$.
    \vspace{0.1cm}
    \item 
    Step C: We improve the estimate of $\left\Vert\tf(t)\right\Vert_{L_{\per}^2}$ by showing the positivity of $f(t)$, $t<T$.
    \vspace{0.1cm}
    \item 
    Step D: We extend $T$ to infinity to derive a global nonnegative and smooth solution. 
\end{itemize}
The detailed proof can be found in what follows. 
\end{proof}
$\bullet$ (Derivation of the equation for the perturbation $\tf$): 
We first write down the weak formulation for $\tf$ as follows:
\begin{equation}\label{eq_tilde(f)}
    \left\la\left\la\pa_{t}\tf, \varphi\right\ra\right\ra+\sigma \left(\pa_{\theta}\tf, \pa_{\theta}\varphi\right)=\kappa \left(L[\tf]\tf,\pa_{\theta}\varphi\right)+\frac{\kappa}{2\pi}\left(L[\tf],\pa_{\theta}\varphi\right)
    \quad \forall \varphi \in H^1_{\per},
    \end{equation}
where 
\[
\tf\in L^2\left(\mathbb{R}_+;H_{\per}^1\right), \ \pa_t \tf\in L^2\left(\mathbb{R}_+; H_{\per}^{-1}\right) \quad \text{and} \quad \tf(t) \to \tilde f_0 \ \text{in} \ L_{\per}^2 \ \text{as}\  t\to 0.
\]

We set
    \[
    T:=\frac{\sigma}{2\kappa^2}\log{\frac{1+2\pi\left\Vert\tilde f_0\right\Vert_{L_{\per}^2}^2}{\frac{1}{4}+2\pi\left\Vert\tilde f_0\right\Vert_{L_{\per}^2}^2}}.
    \]
    In the sequel, we present two preparatory lemmas for local well-posedness.
\begin{lemma}\label{lemma:exist of tf}
 There exists a weak solution $\tf$ on $[0,T)$ satisfying \eqref{eq_tilde(f)}, with $\tf$ uniformly bounded in $L^2\left([0,T);\dot H_{\per}^1\right)\cap H^1\left([0,T);\dot H_{\per}^{-1}\right)$, where $\dot H_{\per}^{k}$ denotes the subspace of $H_{\per}^{k}$ with mean zero.  
    \end{lemma}
\begin{proof}
    For the existence of $\tilde f$, we use the method of Galerkin approximations. 
    Let $\{w_k\}_{k=1}^{\infty}$ be the complete set of normalized eigenfunctions for $-\partial_\theta^2$ in $\dot H_{\per}^1$, and $\{w_k\}_{k=1}^{\infty}$ is an orthonormal basis of $\dot L_{\per}^2$. 
    Suppose that  $\tf^n$ is the unique solution of the following Cauchy problem:
\begin{equation}
		\left\{\begin{aligned}\label{GF-P}	&\pa_t\tf^n+\kappa\pa_\theta\left(L\left[\tf^n\right]\tf^n\right)+\frac{\kappa}{2\pi}\pa_\theta L\left[\tf^n\right]=\sigma\pa_{\theta\theta}\tf^n, \quad \theta\in\s^1,\quad t>0,   \\
        &\tf^n(0)=\Pi_n\left(\tf_0\right),
		\end{aligned}\right.
	\end{equation}
 where $\Pi_n$ is the orthogonal projection on the finite dimensional function space $P_n:=\mathrm{span} \ \{w_k\}_{k=1}^n$.
 
 \vspace{0.15cm}
 Next, we claim that there exists a subsequence $\left\{\tf^{n_k}\right\}_{k=1}^\infty$ satisfying the following three assertions: 
 \begin{itemize}
     \item[(i)] $\tf^{n_k}$ 
     converges weakly to a function $\tf$ in $L^2\left([0,T);\dot H_{\per}^1\right)$.
     \item[(ii)] $\pa_t\tf^{n_k}$ converges weakly to $\pa_t f$ in $L^2\left([0,T); \dot H_{\per}^{-1}\right)$.
     \item[(iii)] $L\left[\tf^{n_k}\right](t)\ \to \ L[\tf](t)$ uniformly for each $t>0$, as $k\to\infty$.
 \end{itemize}
 Note that the weak form of the first equation in \eqref{GF-P} is given by
 \begin{equation}\label{wGF-P}
 \begin{aligned}
     \frac{d}{dt}\left\la\left\la\tf^n,\varphi\right\ra\right\ra+\sigma\left(\pa_\theta \tf^n,\pa_\theta \varphi\right)=\kappa \left(L[\tf^n]\tf^n,\pa_\theta \varphi\right)+\frac{\kappa}{2\pi} \left(L[\tf^n],\pa_\theta \varphi\right),
 \end{aligned} 
 \end{equation}
for all $\varphi\in P_n$. For a fixed $t>0$, we take $\varphi(\theta)=\tf^n(t,\theta)$ in \eqref{wGF-P} to obtain
\begin{equation}\label{eq_tilde fn}
\begin{aligned}
    &\frac{1}{2}\frac{d}{dt}\left\Vert \tf^n(t)\right\Vert^2_{L_{\per}^2}+\sigma\left\Vert\pa_{\theta}\tf^n(t)\right\Vert^2_{L_{\per}^2}\\
    &\hspace{1cm}=\kappa \left(L\left[\tf^n\right](t)\tf^n(t), \pa_{\theta}\tf^n(t)\right)+\frac{\kappa}{2\pi}\left(L\left[\tf^n\right](t), \pa_{\theta}\tf^n(t)\right).
\end{aligned}    
\end{equation}
The difficulty lies in the first term on the right-hand side of \eqref{eq_tilde fn}. We want to control the supremum of $L\left[\tf^n\right]$:
\begin{align*}
    \left\Vert L\left[\tf^n\right](t)\right\Vert_{\infty}
    &=\left\Vert \int_{\s^1}\sin 2(\theta_*-\theta)\tf^n(t,\theta_*)d\theta_*\right\Vert_{\infty}    \leq\int_{\s^1}\left\vert\tf^n(t,\theta_*)\right\vert d\theta_*\\
    &\leq\sqrt{2\pi}\left\Vert \tf^n(t)\right\Vert_{L_{\per}^2}.
\end{align*}
We use the boundedness of $\left\Vert L\left[\tf^n\right](t)\right\Vert_{\infty}$ to estimate two terms on the right-hand side of \eqref{eq_tilde fn}:
\begin{align}\label{new_16}
    &\kappa \left(L\left[\tf^n\right](t)\tf^n(t), \pa_{\theta}\tf^n(t)\right)\notag\\
    \leq &\sqrt{2\pi}\kappa\left\Vert \tf^n(t)\right\Vert^2_{L_{\per}^2}
\left(\tf^n(t),\pa_{\theta}\tf^n(t)\right)\leq \frac{\sigma}{4}\left\Vert \pa_{\theta}\tf^n(t)\right\Vert_{L_{\per}^2}^2+\frac{2\pi\kappa^2}{\sigma}\left\Vert\tf^n(t)\right\Vert_{L_{\per}^2}^4,
\end{align}
and 
\begin{align}\label{new_15}
    &\frac{\kappa}{2\pi}\left(L\left[\tf^n\right](t), \pa_{\theta}\tf^n(t)\right)\notag\\
    &\hspace{0.5cm}\leq\frac{\kappa}{\sqrt{2\pi}}\left\Vert \tf^n(t)\right\Vert^2_{L_{\per}^2}\left(1,\pa_{\theta}\tf^n(t)\right)
    \leq \frac{\sigma}{4}\left\Vert \pa_{\theta}\tf^n(t)\right\Vert_{L_{\per}^2}^2+\frac{\kappa^2}{\sigma}\left\Vert\tf^n(t)\right\Vert_{L_{\per}^2}^2.
\end{align}
Next, we substitute these two estimates into \eqref{eq_tilde fn} to get
\begin{equation}\label{ineq_tilde fn}
    \frac{1}{2}\frac{d}{dt}\left\Vert \tf^n(t)\right\Vert^2_{L_{\per}^2}+\frac{\sigma}{2}\left\Vert\pa_{\theta}\tf^n(t)\right\Vert^2_{L_{\per}^2}\leq \frac{\kappa^2}{\sigma}  \left\Vert \tf^n(t)\right\Vert^2_{L_{\per}^2}\left(1+2\pi \left\Vert \tf^n(t)\right\Vert^2_{L_{\per}^2}\right).
\end{equation}
Note that if $\tilde f_0\equiv0$, then \eqref{eq_tilde(f)} reduces to $\tf(t,\theta)\equiv 0$ for all $t\geq 0$ and $\theta\in \s^1$. 
Thus, we assume 
\[
\left\Vert\tilde f_0\right\Vert_{L_{\per}^2}>0.
\]
Here we choose $n$ large enough such that $\left\Vert \tf^n_0\right\Vert^2_{L_{\per}^2}>0$. Now we drop the second term on the left-hand side of \eqref{ineq_tilde fn} to obtain 
\begin{equation*}
    \frac{1}{2}\frac{d}{dt}\left\Vert \tf^n(t)\right\Vert^2_{L_{\per}^2}\leq \frac{\kappa^2}{\sigma}  \left\Vert \tf^n(t)\right\Vert^2_{L_{\per}^2}\left(1+2\pi \left\Vert \tf^n(t)\right\Vert^2_{L_{\per}^2}\right).
\end{equation*}
This implies 
\[
\left\Vert \tf^n(t)\right\Vert_{L_{\per}^2}\leq \left\Vert \tf^n_0\right\Vert_{L_{\per}^2}\left[\left(1+2\pi\left\Vert \tf^n_0\right\Vert_{L_{\per}^2}^2\right)\exp\left(-\frac{2\kappa^2 t}{\sigma}\right)-2\pi\left\Vert \tf^n_0\right\Vert_{L_{\per}^2}^2\right]^{-\frac{1}{2}},
\]
for 
\[
t<\frac{\sigma}{2\kappa^2}\log\frac{1+2\pi\left\Vert\tilde f_0\right\Vert_{L_{\per}^2}^2}{2\pi\left\Vert\tilde f_0\right\Vert_{L_{\per}^2}^2}.
\]
Then we  obtain
\[
\left\Vert \tf^n(t)\right\Vert_{L_{\per}^2}\leq 2\left\Vert \tf_0\right\Vert_{L_{\per}^2},\quad t\in [0,T). 
\]
By \eqref{ineq_tilde fn}, we have
\[
\left\Vert\pa_{\theta}\tf^n(t)\right\Vert_{L_{\per}^2}\leq C_1\left(\kappa, \sigma, \left\Vert \tf_0\right\Vert_{L_{\per}^2}\right), \quad t\in [0,T).
\] 
This means that $\tf^n$ is bounded in $L^2\left([0,T);\dot H_{\per}^1\right)$, and it follows from \eqref{wGF-P} that 
\begin{align*}
    &\left\la\left\la\pa_t\tf^n(t),\varphi\right\ra\right\ra\\
    &\hspace{1cm}\leq 2\sqrt{2\pi}\kappa\left\Vert \tf_0\right\Vert_{L_{\per}^2} \left\Vert \tf^n(t)\right\Vert_{L_{\per}^2}\left\Vert\pa_{\theta}\varphi\right\Vert_{L_{\per}^2}+\kappa\left\Vert \tf^n(t)\right\Vert_{L_{\per}^2}\Vert\pa_{\theta}\varphi\Vert_{L_{\per}^2}\\
    &\hspace{1.35cm}+\sigma\left\Vert\pa_{\theta}\tf^n(t)\right\Vert_{L_{\per}^2}\Vert\pa_{\theta}\varphi\Vert_{L_{\per}^2}\\
    &\hspace{1cm}\leq C_2\left(\kappa, \sigma, \left\Vert \tf_0\right\Vert_{L_{\per}^2}\right)\Vert\pa_{\theta}\varphi\Vert_{L_{\per}^2}
    \leq C_2\left(\kappa, \sigma, \left\Vert \tf_0\right\Vert_{L_{\per}^2}\right)\Vert \varphi\Vert_{H_{\per}^1}\quad \forall \varphi\in H_{\per}^1.
\end{align*}
Therefore, we have
\begin{equation*} 
\left\Vert\pa_t \tf^n(t)\right\Vert_{H_{\per}^{-1}}\leq C_2\left(\kappa, \sigma, \left\Vert \tf_0\right\Vert_{L_{\per}^2}\right),\quad t\in [0,T), 
\end{equation*}
so that $\tf^n$ is bounded in $H^1\left([0,T);\dot H_{\per}^{-1}\right)$.
Now we use the Ascoli-Arzel\`a theorem to find a subsequence $\{n_k\}_{k=1}^\infty$ such that
\begin{itemize}
    \item[(i)] $\tf^{n_k}$ converges to some $\tf$ in $H^1\left([0,T); \dot H_{\per}^{-1}\right)\cap L^2\left([0,T);\dot H_{\per}^1\right)$, and 
    \item[(ii)]  $L\left[\tf^{n_k}\right]$ converges to some $L$ on $[0,T)$.
\end{itemize}
Since we have for each $\theta\in \s^1$, 
\begin{align*}
	&\left| L\left[\tf^{n_k}\right](t,\theta)-L\left[\tf\right](t,\theta)\right|\\&\hspace{0.5cm}=\int_{S^1}\sin2(\theta_*-\theta)\left(\tf^{n_k}(t,\theta_*)-\tf(t,\theta)\right)d\theta_*\leq \sqrt{2\pi}\left\Vert \tf^{n_k}(t)-\tf(t)\right\Vert_{L_{\per}^2},
\end{align*}
we obtain 
\[
L=L\left[\tf\right].
\]
Finally, we need to show that $\tf$ is indeed a weak solution of \eqref{IBVP} satisfying the initial condition. 
For a fixed $\varphi\in P_n$ in \eqref{wGF-P}, we pass to the weak limit to get 
\begin{equation*}
    \left\la\left\la\pa_t \tf,\varphi\right\ra\right\ra+\sigma\left(\pa_\theta \tf,\pa_\theta \varphi\right)=\kappa\left(L\left[\tf\right]\tf,\pa_\theta \varphi\right)+\frac{\kappa}{2\pi}\left(L\left[\tf\right],\pa_\theta \varphi\right).
\end{equation*}
This is valid for each $n\in \mathbb{N}$. Since $\overline{\bigcup_n P_n}=\dot H_{\per}^1$, $\tf$ is indeed a weak solution of the equation. 
By J. L. Lion's theorem as in \cite[Section 5.9, Theorem 3]{evans2022partial}, we know that actually $\tf\in C\left([0,T);\dot L_{\per}^2\right)$. To complete the proof, we only need to show that $\tf$ satisfies the initial condition. For each $\varphi\in H_{\per}^{1}$,
\begin{align*}
	\left(\tf^n(t)-\Pi_n(f_0),\varphi\right)&=\int_0^t\left(\pa_t \tf^n(\tau),\varphi\right) d\tau\leq C_3\left(\kappa, \sigma, \left\Vert \tf_0\right\Vert_{L_{\per}^2}\right)\left\Vert \varphi\right\Vert_{H_{\per}^{1}}t,
\end{align*} 
uniformly in $n$. So we pass to the limit $t\to 0$ to see $\tf(t)\rightarrow \tf_0$ in $L_{\per}^2$.
\end{proof}

\begin{remark}
	Note that if $f_0\in H_{\per}^m$, $m\in \mathbb{Z}_0$, then the Cauchy problem \eqref{eq_tilde(f)} admits a solution in $L^2\left(\left[0,\tilde T\right);H_{\per}^{m+1}\right)\cap H^1\left(\left[0,\tilde T\right);H_{\per}^{m-1}\right)$, where $\tilde T$ depends on $\kappa$, $\sigma$, $m$ and $\Vert f_0\Vert_{H_{\per}^m}$.
\end{remark}

\begin{lemma}[Continuity with respect to initial data] \label{lemma_continuity_initial}
    Suppose that we have two weak solutions $\tilde f^{(1)}$ and $\tilde f^{(2)}$, with 
    $\left\Vert \tilde f^{(i)}_0\right\Vert_{L^2_{\per}}\leq K$, $i=1,2$. Then $\tilde f^{(1)}-\tilde f^{(2)}$ is bounded in $L^2\left([0,T);\dot H_{\per}^{1}\right)\cap H^1\left([0,T);\dot H_{\per}^{-1}\right)$ by $\ C_4(\kappa, \sigma, K)\left\Vert \tilde f^{(1)}_0-\tilde f^{(2)}_0\right\Vert_{L^2_{\per}}$.
\end{lemma}

\begin{proof}
  We know that $\tilde f^{(1)}$ and $\tilde f^{(2)}$ belong to $L^2\left([0,T);\dot H_{\per}^{1}\right)$. So $\tilde f^{(1)}(t)$ and $\tilde f^{(2)}(t)$ are both in $H_{\per}^{1}$, for a.e. $t\in[0,T)$. Next, we derive the equation for  $h:=\tilde f^{(1)}(t)-\tilde f^{(2)}(t)$. It follows from the definition that for all $\varphi\in H_{\per}^1$,
  \begin{align*}
  	&\langle \la\pa_t h,\varphi\ra\rangle+\sigma(\pa_\theta h,\pa_\theta \varphi)\\
  	&\hspace{1cm}=\kappa \left(L\left[\tilde f^{(1)}\right]\tilde f^{(1)},\pa_\theta \varphi\right)-\kappa \left(L\left[\tilde f^{(2)}\right]\tilde f^{(2)},\pa_\theta \varphi\right)\\
    &\hspace{1.35cm}+\frac{\kappa}{2\pi} \left(L\left[\tilde f^{(1)}\right],\pa_\theta \varphi\right)-\frac{\kappa}{2\pi} \left(L\left[\tilde f^{(2)}\right],\pa_\theta \varphi\right)\\
  	&\hspace{1cm}=\kappa\left(L\left[\tilde f^{(1)}\right]h,\pa_\theta \varphi\right)+\kappa \left((L[h])\tilde f^{(2)},\pa_\theta \varphi\right)+\frac{\kappa}{2\pi} (L[h],\pa_\theta \varphi).
  \end{align*}
  Then we  substitute $h(t)$ into the above equation to get 
  \begin{align*}
  	\langle \la\pa_t h,h\ra\rangle+\sigma(\pa_\theta h,\pa_\theta h)
  	&=\kappa \left(L\left[\tilde f^{(1)}\right]h,\pa_\theta h\right)+\kappa \left(L[h]\tilde f^{(2)},\pa_\theta h\right)+\frac{\kappa}{2\pi}(L[h],\pa_\theta h).
  \end{align*}
  Thus we obtain
  \begin{align}\label{new_10}
  	&\frac{1}{2}\frac{d}{dt}\Vert h(t)\Vert^2_{L^2_{\per}}+\sigma\Vert\pa_{\theta}h(t)\Vert^2_{L^2_{\per}}\notag\\
       &\hspace{0.5cm}\leq \sqrt{2\pi}\kappa\left\Vert \tf^{(1)}(t)\right\Vert_{L^2_{\per}}(h(t),\pa_{\theta}h(t))+\frac{\kappa}{\sqrt{2\pi}}\Vert h(t)\Vert_{L^2_{\per}}(1,\pa_{\theta}h(t))\notag\\
        &\hspace{0.8cm}+\sqrt{2\pi}\kappa\Vert h(t)\Vert_{L^2_{\per}}\left(\tf^{(1)}(t),\pa_{\theta}h(t)\right)\notag \\
  	&\hspace{0.5cm}\leq \frac{1}{2}\sigma\Vert\pa_{\theta}h(t)\Vert^2_{L^2_{\per}}+C_{5}(\sigma,\kappa,K)\Vert h(t)\Vert^2_{L^2_{\per}},\quad t<T.
  \end{align}
  By Gronwall's lemma, the differential inequality \eqref{new_10} yields:
  \begin{equation}\label{3}
      \Vert h(T)\Vert^2_{L^2_{\per}}+\sigma\int_0^T\Vert\pa_{\theta}h(\tau)\Vert^2_{L^2_{\per}}d\tau\leq C_{6}(\sigma,\kappa,K)\Vert h_0\Vert^2_{L^2_{\per}}.
  \end{equation}
Finally, we use \eqref{3} to verify that $h$ is bounded in $L^2\left([0,T);H_{\per}^{1}\right)\cap H^1\left([0,T);H_{\per}^{-1}\right)$ by the constant $C_6\Vert h_0\Vert_{L^2_{\per}}$.
\end{proof}
\begin{remark}
	Lemma \ref{lemma_continuity_initial} implies the uniqueness of the weak solution.
\end{remark}

\begin{proposition}[Classical solutions and nonnegativity]\label{prop:strong solution}
Suppose that $f_0\in H_{\per}^m$ with $m$ large enough such that the weak solution belongs to $C\left([0,T);C^2_{\per}\right)$. Then $f$ is a classical solution. Furthermore, if $f_0$ is nonnegative, then f is nonnegative, and it satisfies the following estimates:
\begin{equation*}
    a(t)\min_{\s^1}f_0<f(t,\theta)<b(t)\max_{\s^1}f_0 \qquad  \forall \  t\in(0,T) \quad \forall \ \theta\in \s^1,
\end{equation*}
where $a=a(t)$, and $b=b(t)$ are explicitly defined as follows:
\begin{equation*}
\left\{\begin{aligned}
    &a(t):=\exp{\left[-2\kappa\int_0^t \left(\sup_{\theta\in\s^1}\int_{\s^1}\cos2(\theta_*-\theta)f(\tau,\theta_*)d\theta_*\right)d\tau\right]},   
    \\
    &b(t):=\exp{\left[2\kappa\int_0^t \left(\sup_{\theta\in\s^1}\int_{\s^1}\cos2(\theta_*-\theta)f(\tau,\theta_*)d\theta_*\right)d\tau\right]}.
\end{aligned}\right.
\end{equation*}
\end{proposition}
\begin{proof}
    Since $f$ lies in $C\left([0,T);C^2_{\per}\right)$ and $D=1$ is assumed, then the right-hand side of \eqref{IBVP} is satisfied in the strong sense. Therefore, the left-hand side of \eqref{IBVP} is also satisfied in the strong sense and it is continuous with respect to $t$ and $\theta$. As a result, $f$ lies in $C^1\left([0,T);C_{\per}\right)$. Thus we can apply the chain rule to $\eqref{IBVP_hom}_1$ to obtain
    \begin{align*}
        \pa_t f(t,\theta)=&\sigma \pa_{\theta\theta}f(t,\theta)-\kappa L[f](t,\theta)\pa_{\theta}f(t,\theta)\\&\hspace{1cm}+2\kappa \left(\int_{\s^1}\cos 2(\theta_*-\theta)f(t,\theta_*)f(t,\theta)d\theta_*\right).
    \end{align*}
    The next part of the proof is the maximum principle. Here we only verify the lower bound. Assume that $f_0$ is positive, and we set $T_0\leq T$ to be the first time that $f(t,\theta)$ reaches $0$ for some $\theta\in\s^1$. Then for all $0\leq t\leq T_0$, 
    \begin{align*}
       \pa_t f(t,\theta)&\geq\sigma \pa_{\theta\theta}f(t,\theta)-\kappa L[f](t,\theta)\pa_{\theta}f(t,\theta)\\
       &\hspace{0.25cm}+2\kappa \left(\sup_{\theta\in\s^1}\int_{\s^1}\cos 2(\theta_*-\theta)f(t,\theta_*)d\theta_*\right)f(t,\theta). 
    \end{align*}
     We set 
     \[
     \Bar{f}(t,\theta):=f(t,\theta)a(t)\min\limits_{\s^1}f_0.
     \]
     Then we have
    \[
    \pa_t \Bar{f}\geq \sigma \pa_{\theta\theta}\Bar{f}-\kappa L[f]\pa_{\theta}\Bar{f}.
    \]
    By the weak maximum principle, we have that for $(t,\theta)\in [0,T_0]\times \s^1$, 
    \[
    \Bar{f}(t,\theta)\geq \min_{\theta_*\in\s^1}\Bar{f}(t,\theta_*)=\min_{\theta_*\in\s^1}\Bar{f}(0,\theta_*).
    \]
    This shows that $T_0=T$ by continuity. For the general case, $f_0$ is only nonnegative and we can do the same analysis for $f_0^{\varepsilon}=\frac{f+\varepsilon}{1+\varepsilon}$ to conclude the desired estimate. 
\end{proof}
\begin{remark}
	By Sobolev's embedding theorem, we only need $m>\frac{5}{2}$ to ensure that $f_0\in C([0,T);C^2_{\per})$.
\end{remark} 
\begin{proposition}[Global existence of weak solutions]\label{prop:uniform in time}
    The local weak solution $\tilde f$ in Lemma \ref{lemma:exist of tf} is indeed a global-in-time solution. 
\end{proposition}
\begin{proof}
    With the well-posedness and positivity of $f$ over $[0,T)$ in Proposition \ref{prop:strong solution}, we can get a better estimate for $\left\Vert \tf(t)\right\Vert_{L_{\per}^2}$:
    \begin{align*}
        \left\vert L\left[\tf\right](t,\theta)\right\vert=&\vert L[f](t,\theta)\vert
        =\left\vert\int_{\s^1}\sin2(\theta_*-\theta)f(t,\theta)d\theta_*\right\vert
        \leq\int_{\s^1}f(t,\theta_*)d\theta_*
        =1.
    \end{align*}
    We return to the estimate of $\tf$:
    \begin{align*}
        \frac{1}{2}\frac{d}{dt}\left\Vert \tf(t)\right\Vert^2_{L_{\per}^2}+\sigma\left\Vert \pa_{\theta}\tf(t)\right\Vert^2_{L_{\per}^2}=\kappa \left(L\left[\tf\right]\tf,\pa_{\theta}\tf\right)+\frac{\kappa}{2\pi}\left(L\left[\tf\right],\pa_{\theta}\tf\right).
    \end{align*}
    For the first term on the right-hand side, we improve the estimate in \eqref{new_16}
    \begin{equation*}
       \kappa \left(L\left[\tf\right]\tf,\pa_{\theta}\tf\right)\leq \frac{\sigma}{4}\left\Vert \pa_{\theta}\tf(t)\right\Vert^2_{L_{\per}^2}+\frac{\kappa^2}{\sigma}\left\Vert \tf(t)\right\Vert^2_{L_{\per}^2}.
    \end{equation*}
    For the second term on the right-hand side, we use the previous estimate in \eqref{new_15}. Then we have 
    \begin{equation*}
        \frac{1}{2}\frac{d}{dt}\left\Vert \tf(t)\right\Vert^2_{L_{\per}^2}+\frac{\sigma}{2}\left\Vert \pa_{\theta}\tf(t)\right\Vert^2_{L_{\per}^2}\leq \frac{\kappa^2}{\sigma}(1+2\pi)\left\Vert \tf(t)\right\Vert^2_{L_{\per}^2}.
    \end{equation*}
    We use Gronwall's lemma to see
    \begin{equation}\label{est of tilde f v2}
        \left\Vert \tf(t)\right\Vert^2_{L_{\per}^2}\leq \exp{\left(\frac{\kappa^2}{\sigma}(1+2\pi)t\right)}\left\Vert \tf_0\right\Vert^2_{L_{\per}^2}\quad \forall \ t<T.
    \end{equation}
    We use an induction to define the solution as follows. We set
    \begin{equation}\label{def_T_k}
    T_1=T \quad \text{and}\quad T_{k+1}-T_k=\frac{\sigma}{2\kappa^2}\log{\frac{1+2\pi\left\Vert\tilde f(T_k)\right\Vert_{L_{\per}^2}^2}{\frac{1}{4}+2\pi\left\Vert\tilde f(T_k)\right\Vert_{L_{\per}^2}^2}}, \quad k\geq 1.
    \end{equation}
    Then it follows from \eqref{est of tilde f v2} that
    \[
    \left\Vert \tf(T_k)\right\Vert^2_{L_{\per}^2}\leq \exp{\left(\frac{\kappa^2}{\sigma}(1+2\pi)T_k\right)}\left\Vert \tf_0\right\Vert^2_{L_{\per}^2}.
    \]
    We prove the proposition by the contradiction argument. If we cannot extend $T_k$ to infinity, then the sequence $\{T_k\}_{k}$ has a finite limit, that is to say, $\left\Vert \tf(T_k)\right\Vert^2_{L_{\per}^2}$ has a upper bound for all $k\geq 1$. It results in that $T_{k+1}-T_k$ has a lower bound for all $k$ by \eqref{def_T_k}, which gives a contradiction.  
\end{proof}
\begin{proposition}[Regularity]\label{prop: smoothness}
  Suppose that $f_0$ is a probability density function on $\s^1$,and it belongs to $H_{\per}^m$ with $m>\frac{5}{2}$. Then, the weak solution $\tf$ in Proposition \ref{prop:uniform in time} belongs to $C^{\infty}(\mathbb{R}_+\times \s^1)$.  
\end{proposition}
\begin{proof}
    Since $m>\frac{5}{2}$, $\tf$ belongs to $L_{\per}^2$. Thus, $\tf$ belongs to $L^2\left(\mathbb{R}_+;\dot H_{\per}^{1}\right)$. For all $t>0$, there exists $0<s<t$ such that $\tf(s)\in \dot H_{\per}^{1}$. So with the initial data $\tf(s)$, we are able to construct a solution, which belongs to $L^2\left([s,\infty);\dot H_{\per}^{2}\right)\cap H^1\left([s, \infty);\dot L_{\per}^{2}\right)$. Besides, this solution is also a weak solution in $L^2\left(\mathbb{R}_+;\dot H_{\per}^{1}\right)\cap H^1\left(\mathbb{R}_+;\dot H_{\per}^{-1}\right)$. By the uniqueness, it is indeed the solution $f$. We can repeat this procedure to improve the regularity such that $\tf\in C^1\left(\mathbb{R}_+;\dot H_{\per}^m\right)$ for any $m\in \mathbb{Z}_0$. Then we differentiate it with respect to the time variable and use the Sobolev embedding to see that it is smooth with respect to $t$ and $\theta$. 
\end{proof}
Now we are ready to provide the proof of Theorem \ref{thm:pde}.
\begin{proof}[Proof of Theorem \ref{thm:pde}:]$\ $
    By Lemma \ref{lemma:exist of tf}, we have known that there exists a weak solution in the space $L^2\left([0,T);\dot H_{\per}^1\right)\cap H^1\left([0,T);\dot H_{\per}^{-1}\right)$. Therefore, the existence of $f$ can be obtained in $L^2\left([0,T); H_{\per}^1\right)\cap H^1\left([0,T); H_{\per}^{-1}\right)$ with unit mass. The uniqueness can be gained from Lemma \ref{lemma_continuity_initial}. Since $m>\frac{5}{2}$, using the maximal principle in Proposition \ref{prop:strong solution}, the nonnegativity can be proved on $[0,T)$, which is used to improve the estimation in Proposition \ref{prop:uniform in time} to get the solution in the whole time interval. The smoothness is shown in Proposition \ref{prop: smoothness}.
\end{proof}
\section{Phase Transition for constant noise}
In this section, we study phase transition for space-homogeneous case with constant noise $D\equiv 1$. We first study equilibrium as well as the free energy, and we study the convergence toward equilibrium in a spatially homogeneous setting.
\subsection{Equilibrium and free energy}
In this section, we study the equilibrium and the free energy of system \eqref{IBVP_hom}. Since we have the global well-posedness and regularity of $f$, we assume that $f$ is smooth and positive for all $(t,\theta)\in \R_+\times \s^1$.

\vspace{0.15cm}
Note that the equation $\eqref{IBVP_hom}_1$ can be rewritten as 
\begin{equation}\label{Kinetic_constD_log}
    \pa_t f=\pa_\theta\left[f\left(\sigma \pa_\theta\log f-\kappa L[f]\right)\right],
\end{equation}
and we define the free energy as
\begin{equation}\label{free_energy}
    \mathcal{F}[f](t)=\sigma\int_{\s^1}f(t,\theta)\log f(t,\theta)d\theta-\frac{\kappa}{4} \int _{\s^1\times\s^1} f(t,\theta)f(t,\theta_*) \cos 2(\theta_*-\theta) d\theta_*d\theta.
\end{equation}
In the following proposition, we show that the energy functional in \eqref{free_energy} is non-increasing in time.
\begin{proposition}\label{prop_free_energy_const}
Let $f$ be a classical solution to \eqref{Kinetic_constD_log}. 
Then we have
\begin{equation*}
    \frac{d}{dt}\mathcal{F}[f](t)=-\int_{\s^1} f\left(\sigma \pa_\theta\log f-\kappa L[f]\right)^2 d\theta \\
    =:-\mathcal{D}[f](t)\leq 0,\quad t>0.
\end{equation*}    
\end{proposition}
\begin{proof}
We first differentiate \eqref{prop_free_energy_const} with respect to $t$ to obtain
\begin{align*}
 \frac{d}{dt}\mathcal{F}[f](t)= &\int_{\s^1} \pa_t f(t,\theta)\left(\sigma \log f(t,\theta)-\frac{\kappa}{2} \int _{\s^1} f(t,\theta_*) \cos 2(\theta_*-\theta) d\theta_*\right)d\theta\\
 =&\int_{\s^1} \frac{\pa}{\pa \theta}\left[f\left(\sigma \frac{\pa}{\pa\theta}\log f-\kappa L[f]\right)\right]\left(\sigma \log f-\kappa L[f]\right)d\theta\\
    =&-\int_{\s^1} f\left(\sigma \pa_\theta\log f-\kappa L[f]\right)^2 d\theta \leq 0.
\end{align*}
\end{proof}
Similar to \cite[Proposition 3.2]{frouvelle2012dynamics},  we can derive the LaSalle invariance principle which is stated below without a proof.
\begin{lemma}[The LaSalle invariance principle]
    Suppose that $f_0$ is a probability density function on $\s^1$, and we denote by $\mathcal{F}_{\infty}$ the constant limit of $\mathcal{F}[f](t)$ as $t\to \infty$, and let $f$ be the classical solution of \eqref{IBVP_hom}. Then we have 
    \[
    \mathcal{E}_\infty=\{f\in \mathit C^\infty(\s^1) |  \ \mathcal{D}[f]\equiv0, \ \mathcal{F}[f]=\mathcal{F}_{\infty}\}\neq\emptyset.
    \]
\end{lemma}

\vspace{0.15cm}
In the statement of the LaSalle invariance principle, we may see heuristically that there is a convergence of $\mathcal F[f](t)$. Thus, the equilibrium  should satisfy the equation $\mathcal D[f]=0$. This leads to the von-Mises type equilibrium of the following form:
\begin{equation}\label{von_Mises_eta_Phi}
    M_{\eta, \Phi}(\theta)=\frac{1}{Z(\eta)}\exp\big(\eta\cos2(\Phi-\theta)\big),
\end{equation}
where $Z$ is the normalization constant such that $M_{\eta, \Phi}$ is a probability density function on $\s^1$:
\[
Z(\eta) = \int_{\mathbb{S}^1}\exp\big(\eta\cos2(\Phi-\theta)\big)\,d\theta.
\]
Note that $Z$ is independent of $\Phi$ due to the periodicity of trigonometric functions.
By Proposition \ref{prop_free_energy_const}, the necessary condition for $M_{\eta,\Phi}$ to be an equilibrium in $\mathcal{E}_\infty$ is
\[
\sigma \pa_\theta\log M_{\eta,\Phi}(\theta)-\kappa L[M_{\eta,\Phi}](\theta)=0 \quad \forall \ \theta\in\s^1.
\]
This results in the so-called ``\textit{compatibility condition}":

\begin{equation} \label{new_14}
\int_{\s^1}M_{\eta,\Phi}(\theta)\cos 2(\theta-\Phi)\,d\theta = \frac{2\sigma}{\kappa}\eta.
\end{equation}
Therefore, without loss of generality, we set $\Phi = 0$, and use the following form in the sequal:
\begin{equation*}
    M_{\eta}(\theta)=M_{\eta,0}(\theta)=\frac{1}{Z(\eta)}\exp\big(\eta\cos2\theta\big).
\end{equation*}
In this case, the compatibility condition \eqref{new_14} becomes
\begin{equation}\label{compatibility}
\int_{\s^1}M_{\eta}(\theta)\cos 2\theta\,d\theta = \frac{2\sigma}{\kappa}\eta.
\end{equation}
In what follows, we study the large-time behavior and the phase transition of the system with respect to the parameters $\sigma$ and $\kappa$. 
In other words, we analyze the von-Mises type equilibria $M_{\eta}$ generated by the nonnegative triple solutions $(\sigma, \kappa, \eta)$ to the compatibility condition \eqref{compatibility}, and their convergence behaviors. We first introduce the ``\textit{weighted average}'' associated to $M_{\eta}$ given by 
\[
\langle \gamma(\cdot)\rangle_M \equiv\int_{\s^1}\gamma(\theta) M_{\eta}(\theta)  d\theta,
\]
where $\gamma(\theta)$ is some given periodic function on $\s^1$. 
Then we have the following lemma.
\begin{lemma}\label{lemma_kappa_sigma}
    Suppose that parameters $\sigma$, $\kappa$ and $\eta$ satisfy 
    \[
    0<\frac{\sigma}{\kappa}<\frac{1}{4},\quad\eta>0,
    \]
    and the compatibility condition \eqref{compatibility} holds. Then we have
    \begin{equation*}
        \frac{2\sigma}{\kappa}=\left\langle \sin^2 2\theta \right\rangle _M>\left\langle \cos^2 2\theta \right\rangle _M-\left\langle \cos 2\theta \right\rangle _M^2.
    \end{equation*}
\end{lemma}
\begin{proof}
By integration by parts, we have
\begin{align*}
&\int_{\s^1}\exp\left(\eta\cos2\theta\right)\sin^2 2\theta d\theta \\
&\hspace{0.5cm}= -\frac{1}{2}\int_{\s^1}\exp\left(\eta\cos2\theta\right)\sin 2\theta \,d\cos 2\theta 
= -\frac{1}{2\eta}\int_{\s^1}\sin 2\theta \,d\exp\left(\eta\cos2\theta\right) \\
&\hspace{0.5cm}=\frac{1}{2\eta}\int_{\s^1}\exp\left(\eta\cos2\theta\right) \,d\sin 2\theta 
=\frac{1}{\eta}\int_{\s^1}\exp\left(\eta\cos2\theta\right) \cos 2\theta\,d\theta.
\end{align*}
Thanks to the compatibility condition $\eqref{compatibility}$, we have
\[
\left\langle \sin^2 2\theta \right\rangle _M = \frac{1}{\eta}\langle \cos 2\theta \rangle _M = \frac{2\sigma}{\kappa}.
\]
Next, we consider the function $k(\eta)$ given by
\[
k(\eta) = \int_{\s^1}e^{\eta\cos2\theta}\left(\cos^2 2\theta-\sin^22\theta\right)\,d\theta\int_{\s^1}e^{\eta\cos2\theta}\,d\theta 
- \left(\int_{\s^1}e^{\eta\cos2\theta}\cos 2\theta \,d\theta\right)^2.
\]
Note that $k(0)=0$, and by direct calculation, we have
\begin{align*}
k'(\eta) = &\int_{\s^1}e^{\eta\cos2\theta}\cos 2\theta\left(\cos^2 2\theta-\sin^22\theta\right)\,d\theta\int_{\s^1}e^{\eta\cos2\theta}\,d\theta \\
&+\int_{\s^1}e^{\eta\cos2\theta}\left(\cos^2 2\theta-\sin^22\theta\right)\,d\theta\int_{\s^1}e^{\eta\cos2\theta}\cos2\theta\,d\theta \\
&- 2\int_{\s^1}e^{\eta\cos2\theta}\cos 2\theta \,d\theta
\int_{\s^1}e^{\eta\cos2\theta}\cos^2 2\theta \,d\theta\\
=&-2\int_{\s^1}e^{\eta\cos2\theta}\cos 2\theta\sin^22\theta\,d\theta\int_{\s^1}e^{\eta\cos2\theta}\,d\theta \\
=&\frac{1}{\eta}\int_{\s^1}\cos 2\theta\sin2\theta\,de^{\eta\cos2\theta}\int_{\s^1}e^{\eta\cos2\theta}\,d\theta \\
=&-\frac{2}{\eta}\int_{\s^1}e^{\eta\cos2\theta}\left(1-2\sin^22\theta\right)\,d\theta\int_{\s^1}e^{\eta\cos2\theta}\,d\theta \\
=&-\frac{2}{\eta}Z^2(\eta)\left\langle1-2\sin^22\theta\right\rangle_M\\
=&-\frac{2}{\eta}Z^2(\eta)\left(1 - 4\frac{\sigma}{\kappa}\right)<0 \quad \text{ since } \; 0<\frac{\sigma}{\kappa}<\frac{1}{4}.
\end{align*}
This implies 
\[
k(\eta)<k(0)=0,\quad \eta>0.
\]
and we have
\[
\left\langle \sin^2 2\theta \right\rangle _M>\left\langle \cos^2 2\theta \right\rangle _M-\langle \cos 2\theta \rangle _M^2.
\]
\end{proof}

We are now ready to state our main theorem on the phase transition of the equilibrium. 
\begin{theorem}[Phase transition]\label{thm_unif_eta}
There are only two types of nonnegative solutions to the compatibility condition \eqref{compatibility}:
    \begin{enumerate}
        \item If $\ \displaystyle\frac{\sigma}{\kappa}\geq \frac{1}{4}$, then the only solution is $\eta=0$, 
        and the von-Mises function reduces to a constant function $\displaystyle M_0(\theta)=f_{\infty}^c (\theta)\equiv \frac{1}{2\pi}$.
        \item  If $\ 0 < \dfrac{\sigma}{\kappa}<\dfrac{1}{4}$, then there are two solutions: one is the trivial solution $\eta=0$, and the other is $\eta>0$. 
        \end{enumerate}
\end{theorem}
\begin{proof}
Note that $\eta = 0$ with arbitrary $\frac{\sigma}{\kappa} > 0$ is always a solution to the compatibility condition \eqref{compatibility}.
Next we will show that this is the only solution for $\frac{\sigma}{\kappa}\geq\frac14$, and there is another positive solution $\eta$ for $0 < \frac{\sigma}{\kappa}<\frac14$. Define $\beta=\beta(\eta)$ as follows:
\begin{equation*}
\beta(\eta):=\displaystyle\left\{
\begin{array}{cl}
\displaystyle \frac{\langle \cos 2\theta \rangle_M}{\eta}, &\text{ if }\eta\in(0,\infty),\\
\displaystyle \frac{\int_{\s^1}\cos^22\theta d\theta}{\int_{\s^1}d\theta}=\frac{1}{2}, &\text{ if } \eta = 0.
\end{array}\right.
\end{equation*}
One can show that $\beta$ is continuous for $\eta \in [0, \infty)$,
and for given $\frac{\sigma}{\kappa}$, the compatibility condition \eqref{compatibility} is equivalent to that the following equation:
\begin{equation}\label{compatibility_beta}
\beta(\eta) = \frac{\langle \cos 2\theta \rangle_M}{\eta}= \frac{2\sigma}{\kappa}, \quad  \eta > 0,   
\end{equation}
is solvable for a positive solutions $\eta$ for some given $\frac{\sigma}{\kappa}$.
We study the derivative $\beta'(\eta)$ 
to obtain the correspondence between $\eta$ and $\frac{\sigma}{\kappa}$: 
\begin{align*}
    \beta'(\eta)=&\frac{1}{\eta^2}\left(\eta\left\langle \cos^2 2\theta\right\rangle_M - \langle \cos 2\theta\rangle_M\langle 1+\eta \cos 2\theta\rangle_M\right)\\
    =&\frac{1}{\eta^2}\left(\eta-2\langle \cos 2\theta\rangle_M - \eta\langle \cos 2\theta\rangle_M^2\right) \\
    =& \frac{1}{\eta}\left[1 - \frac{4\sigma}{\kappa} - \left(\frac{2\sigma}{\kappa}\eta\right)^2\right] \quad \text{ due to \eqref{compatibility_beta}.}
\end{align*}
Next, we consider two cases depending on $\frac{\sigma}{\kappa}$:

\noindent$\bullet$ Case A (subcritical case): If $\displaystyle\frac{\sigma}{\kappa}\geq \frac{1}{4}$, then we must have $\beta'(\eta)<0$ for all $\eta>0$. This results in 
\[
\beta(\eta)<\beta(0)=\displaystyle\frac{1}{2}\leq \displaystyle\frac{2\sigma}{\kappa}.
\]
This shows that there is only the trivial solution $\eta=0$ for the subcritical case $\displaystyle\frac{\sigma}{\kappa}\geq \frac{1}{4}$.

\noindent$\bullet$ Case B (supercritical case): If $\displaystyle 0 < \frac{\sigma}{\kappa}< \frac{1}{4}$, by Lemma \ref{lemma_kappa_sigma}, we have
\begin{align*}
    \beta'(\eta)& = \frac{1}{\eta^2}\left[\eta\left(\left\langle \cos^2 2\theta\right\rangle_M - \langle \cos 2\theta\rangle_M^2\right) - \langle \cos 2\theta\rangle_M\right]\\
    &< \frac{1}{\eta^2}\left[\eta\left\langle \sin^2 2\theta\right\rangle_M - \langle \cos 2\theta\rangle_M\right] = 0.
\end{align*}
Meanwhile, by the definition of $\beta=\beta(\eta)$, we have
\[
\lim_{\eta\to \infty} \beta(\eta) = 0.
\]
Thus, there exists a unique positive solution $\eta>0$ if and only if $0<\frac{\sigma}{\kappa}<\frac{1}{4}$. 
\end{proof}

\subsection{Convergence in the subcritical regime}

In this section, we focus on the large-noise case and study the convergence of the classical solution to \eqref{IBVP_hom} toward the unique steady state. Combining with the LaSalle invariance principle, the convergence is obtained using the free energy, since the invariant measure is unique. To study the convergence rate, we consider both in $L_2$ case and the evolution of relative entropy. We first introduce the relative entropy of two probability density functions $f$ and $h$ where $f$ is absolutely continuous with respect to $h$:
\begin{equation}\label{relative entropy}
    \mathcal{H}\left[f|h\right](t):=\int_{\s^1}f(t,\theta)\log\left(\frac{f(t,\theta)}{h(t,\theta)}\right)d\theta.
\end{equation}
\begin{proposition}
    Let $f$ be the classical solution of \eqref{IBVP_hom}. Then the following assertions hold:
    \begin{enumerate}
        \item [(i)]If $\frac{\sigma}{\kappa}>2$, then the relative entropy decays to zero exponentially fast:
        \begin{equation*}
        \mathcal{H}\left[f|f_{\infty}^c\right](t)\leq e^{-2\kappa\left(\frac{\sigma}{\kappa}-2\right)t}\mathcal{H}\left[f_0|f_{\infty}^c\right]\quad \text{as} \quad t\to\infty.
        \end{equation*}
         \item [(ii)]If $\frac{\sigma}{\kappa}>3$, then $\Vert f(t)-f_{\infty}^{\rm c}\Vert_{L_{\per}^2}$ decays to zero exponentially fast:
         \begin{equation*}
             \Vert f(t)-f_{\infty}^{\rm c}\Vert_{L_{\per}^2}\leq e^{-3\kappa\left(\frac{\sigma}{\kappa}-\frac{1}{3}\right)t}\left\Vert f_0-f_{\infty}^{\rm c}\right\Vert_{L_{\per}^2}\quad \text{as} \quad t\to\infty.
         \end{equation*}
    \end{enumerate}
\end{proposition}
\begin{proof}
     (i). We differentiate $\mathcal{H} (f|f_{\infty}^c)$ in \eqref{relative entropy}:
    \begin{align}\label{new_5}
        \frac{d}{dt}\mathcal{H}\left[f|f_{\infty}^{\rm c}\right](t)
        =&\int_{\s^1}\pa_t f(t)\log f(t)d\theta\notag\\=&\int_{\s^1}\pa_\theta\left[f(t)\pa_\theta\Big(\sigma \log f(t)-\kappa L[f](t)\Big)\right]\log f(t)d\theta\notag\\
        =&-\sigma \int_{\s^1}\left|\pa_\theta\log f(t)\right|^2f(t)d\theta+4\kappa\int_{\s^1} L[f]f(t)d\theta\notag\\
        =&:\mathcal{I}_{11}(t)+\mathcal{I}_{12}(t).
    \end{align}
    Next, we estimate $\mathcal{I}_{1i}$, $i=1,2$ one by one.
    
    \vspace{0.15cm}
    \noindent$\bullet$ Case A.1 (Estimate of $\mathcal{I}_{11}$):
    By the Logarithmic Sobolev inequalities (see Theorem 3.1 in \cite{gentil2004inegalites}), 
    \begin{equation*}
        \mathcal{H}[f|f_{\infty}^{\rm c}](t)\leq \frac{1}{2}\int_{\s^1}\left|\pa_\theta\log f(t)\right|^2f(t)d\theta. 
    \end{equation*}
    That is to say,
    \begin{equation}\label{new_3}
        \mathcal{I}_{11}(t)\leq-2\sigma\mathcal{H}[f|f_{\infty}^{\rm c}](t).
    \end{equation}
    \newline
    \noindent$\bullet$ Case A.2 (Estimate of $\mathcal{I}_{12}$):
    For the second term $\mathcal{I}_{12}$, we use the Csisz{\'a}r-Kullback-Pinsker inequality to obtain
    \begin{align*}
        \int_{\s^1} L[f]fd\theta&=\int_{\s^1\times\s^1} \sin 2(\theta_*-\theta)f(t,\theta_*)f(t,\theta)d\theta_*d\theta \leq \Vert f(t)-f_{\infty}^{\rm c}\Vert_{L_{\per}^1}^2\\
        &\leq 2\mathcal{H}\left[f|f_{\infty}^{\rm c}\right](t),
    \end{align*}
    i.e.,
    \begin{equation}\label{new_4}
        \mathcal{I}_{12}(t)\leq 4\kappa\mathcal{H}\left[f|f_{\infty}^{\rm c}\right](t).
    \end{equation}
    In \eqref{new_5}, we combine \eqref{new_3} and \eqref{new_4} to get
\begin{equation*}
    \frac{d}{dt}\mathcal{H}[f|f_{\infty}^{\rm c}](t)\leq -2\kappa\left(\frac{\sigma}{\kappa}-2\right)\mathcal{H}[f|f_{\infty}^{\rm c}](t),
\end{equation*}
and we use Gronwall's lemma to derive the assertion (i).

\vspace{0.15cm}
\noindent(ii). We define the perturbation $\tilde f(t,\theta):=f(t,\theta)-f_{\infty}^{\rm c}(\theta)$. 
For $L_{\rm per}^2$-estimate, we multiply \eqref{IBVP_hom} by $\tilde f$ and integrate the resulting relation on $\s^1$ to obtain
\begin{align}\label{new_6}
    \frac{1}{2}\frac{d}{dt}\left\Vert \tilde f(t)\right\Vert _{L_{\per}^2}^2=&-\sigma\left\Vert \pa_\theta \tilde f(t)\right\Vert _{L_{\per}^2}^2+\kappa\int_{\s^1}L\left[\tilde f\right](t)\tilde f(t)\pa_\theta\tilde f(t)d\theta\notag\\
    &+\frac{\kappa}{2\pi} \int_{\s^1}L\left[\tilde f\right](t)\pa_\theta\tilde f(t)d\theta\notag\\=&:\mathcal{I}_{21}(t)+\mathcal{I}_{22}(t)+\mathcal{I}_{23}(t).
\end{align}
Next, we estimate the terms $\mathcal{I}_{2i}$, $i=1,2,3$ separately. 

\vspace{0.15cm}
\noindent$\bullet$ Case B.1 (Estimate of $\mathcal{I}_{21}$):
By the Poincare inequality, for $\tilde f(t)\in L_{\per}^2$,
\begin{equation*}
    \left\Vert \pa_\theta \tilde f(t)\right\Vert _{L_{\per}^2}^2\geq \left\Vert \tilde f(t)\right\Vert _{L_{\per}^2}^2.
\end{equation*}
Therefore, we have
\begin{equation}\label{new_7}
    \mathcal{I}_{21}(t)\leq -\sigma\left\Vert\tilde f(t)\right\Vert _{L_{\per}^2}^2.
\end{equation}
\newline
\noindent$\bullet$ Case B.2 (Estimate of $\mathcal{I}_{22}$):
Note that 
\[
\int_{\s^1}L\left[\tilde f\right]\tilde f\pa_\theta\tilde fd\theta=-\int_{\s^1}L\left[\tilde f\right]\tilde f\pa_\theta\tilde fd\theta-\int_{\s^1}\pa_\theta L\left[\tilde f\right]\tilde f^2d\theta.
\]
This yields
\begin{align*}
    \mathcal{I}_{22}(t)=&\kappa\int_{\s^1}L\left[\tilde f\right](t)\tilde f(t)\pa_\theta\tilde f(t)d\theta=-\frac{\kappa}{2}\int_{\s^1}\pa_\theta L\left[\tilde f\right](t)\tilde f^2(t)d\theta\notag\\=&\kappa\int_{\s^1\times \s^1}\cos 2(\theta_*-\theta)\tilde f(t,\theta_*)\tilde f^2(t,\theta)d\theta d\theta_*.
\end{align*}
Since $f$ is nonnegative and has the unit mass on $\s^1$, we have
\begin{equation*}
    \left\vert\int_{\s^1}\cos 2(\theta_*-\theta)\tilde f(t,\theta_*)\tilde d\theta_*\right\vert_{\infty}\leq 1.
\end{equation*}
Therefore, we have 
\begin{equation}\label{new_8}
    \mathcal{I}_{22}\leq\kappa\int_{\s^1\times \s^1}\cos 2(\theta_*-\theta)\tilde f(t,\theta_*)\tilde f^2(t,\theta)d\theta d\theta_*\leq\kappa \left\Vert \tilde f(t)\right\Vert _{L_{\per}^2}^2.
\end{equation}
\newline
\noindent$\bullet$ Case B.3 (Estimate of $\mathcal{I}_{23}$):
By direct calculation, we have 
\begin{align}\label{new_9}
    \mathcal{I}_{23}=&\frac{\kappa}{2\pi}\int_{\s^1}L\left[\tilde f\right](t,\theta)\pa_\theta\tilde f(t,\theta)d\theta=-\frac{\kappa}{2\pi}\int_{\s^1}\pa_\theta L\left[\tilde f\right](t,\theta)\tilde f(t,\theta)d\theta\notag\\=&\frac{\kappa}{\pi}\int_{\s^1\times \s^1}\cos 2(\theta_*-\theta)\tilde f(t,\theta_*)\tilde f(t,\theta)d\theta d\theta_*\leq\frac{\kappa}{\pi}\left\Vert \tilde f(t)\right\Vert _{L_{\per}^1}^2.
\end{align}
Therefore, we combine \eqref{new_7} - \eqref{new_9} to get the estimate of \eqref{new_6}:

\begin{align*}
    &\frac{1}{2}\frac{d}{dt}\left\Vert \tilde f(t)\right\Vert _{L_{\per}^2}^2\\
    &\hspace{1cm
    }\leq-\sigma\left\Vert\tilde f(t)\right\Vert _{L_{\per}^2}^2+\frac{\kappa}{\pi}\left\Vert \tilde f(t)\right\Vert _{L_{\per}^1}^2+\kappa \left\Vert \tilde f(t)\right\Vert _{L_{\per}^2}^2
    \leq (-\sigma+3\kappa) \left\Vert \tilde f(t)\right\Vert _{L_{\per}^2}^2,
\end{align*}
where the last step is achieved by the following inequality:
\begin{equation*}
    \left\Vert \tilde f(t)\right\Vert_{L_{\per}^1}^2\leq2\pi\left\Vert \tilde f(t)\right\Vert_{L_{\per}^2}^2.
\end{equation*}
As a result, we obtain the following inequality:
\begin{equation*}
    \frac{d}{dt}\left\Vert \tilde f(t)\right\Vert _{L_{\per}^2}^2\leq -2\kappa\left(\frac{\sigma}{\kappa}-3\right)\left\Vert \tilde f(t)\right\Vert _{L_{\per}^2}^2,
\end{equation*}
 and use the Gronwall's lemma to derive the desired estimate. 
\end{proof}

\subsection{Convergence in the supercritical regime}

In order to obtain a quantitative description for the state of alignment and the convergence to equilibria, we introduce the order parameter $(r_n,\phi_n)$ for the mean-field systems analogous to the particle system:
\begin{definition}
Let $f=f(t,\theta)$ be the one-particle distribution function in \eqref{IBVP_hom}. Then the real-valued order parameters $\big(r_n(t),\phi_n(t)\big)$ for \eqref{IBVP_hom} is given by
\begin{equation}\label{def_rn_phin}
	r_n(t)e^{\mathrm{i} n\phi_n(t)}=\int_{ \s^1}f(t,\theta)e^{\mathrm{i} n\theta}d\theta,
\end{equation}
where
\begin{equation}\label{def_rn}
	r_n(t)=\left\vert\int_{ \s^1}f(t,\theta)e^{\mathrm{i} n\theta}d\theta\right\vert.
\end{equation}
\end{definition}
\begin{lemma}
Let $(r_n,\phi_n)$ be defined as in \eqref{def_rn_phin} and \eqref{def_rn}. Then they satisfy
\begin{equation}\label{eqn_rn}
    \left\{\begin{aligned}
        &r_n(t)=\int_{\s^1}f(t,\theta)\cos n(\theta-\phi_n(t))d\theta, \qquad  t>0,	\\
&0=\int_{\s^1}f(t,\theta)\sin n(\theta-\phi_n(t))d\theta.
    \end{aligned}\right.
\end{equation}
\end{lemma}
\begin{proof}
We multiply the both sides of \eqref{def_rn_phin} by $e^{-\mathrm{i}n\phi_n(t)}$ to get
\begin{equation*}
r_n(t)=\int_{\mathbb{S}^1}f(t,\theta)e^{\mathrm{i} n(\theta-\phi_n(t))}d\theta.
\end{equation*}
Now we take the real and imaginary parts to get the desired estimates.
\end{proof}
For the nematic alignment, we focus on the order parameters with $n=2$. 
\begin{proposition}\label{prop_supercritical}
The following assertions hold:
\begin{enumerate}
    \item If $r_2(0)=0$, then $f$ converges to the uniform distribution $f_{\infty}^{\rm c}=\frac{1}{2\pi}$ exponentially fast.
    \vspace{0.2cm}
    \item If $r_2(0) > 0$, then 
    the stable equilibrium takes the form of $M_{\eta,\Phi}(\theta)$ in \eqref{von_Mises_eta_Phi} with $\eta>0$ and $\Phi$ satisfying the compatibility condition \eqref{compatibility}. Furthermore, for each $m\in \mathbb{Z}_0$,
    \begin{equation}\label{converge}
        \lim_{t\to\infty}\Vert f(t)-M_{\eta,\phi_2(t)}\Vert_{H_{\rm per}^m}=0.
    \end{equation}
\end{enumerate}
\end{proposition}
\begin{proof}
Recall that
\[
r_2(t)=\int_{\mathbb{S}^1}f(t,\theta)e^{2\mathrm{i} (\theta-\phi_2(t))}d\theta.
\]
Consider the time evolution of $r_2(t)$:    
    \begin{align}
            \frac{dr_2(t)}{dt}=&\int_{\s^1}f(t,\theta)\left[-2\kappa L[f]\sin 2(\theta-\phi_2(t))-4\sigma \cos 2(\theta-\phi_2(t))\right]d\theta \notag \\
            =&-4\sigma r_2(t)+2\kappa r_2(t)\int_{\s^1}f(t,\theta)\sin^2 2(\theta-\phi_2(t))d\theta\notag\\ =&2\kappa r_2(t)\left[\int_{\s^1}f(t,\theta)\sin^2 2(\theta-\phi_2(t))d\theta-\frac{2\sigma}{\kappa}\right].\label{d2r2}
    \end{align}    
    \noindent(1) If $r_2(0)=0$, then $r_2(t)=0$ for all $t> 0$ by \eqref{d2r2}.  
    Then we have 
    \[
    L[f](t,\theta)=0 \quad \forall t>0 \quad\forall\ \theta\in \s^1,
    \]
    so that \eqref{IBVP_hom} reduces to the heat equation on the torus. The equilibrium $f_{\infty}$ must be the uniform distribution $f_{\infty}^{\rm c}$ and we obtain an exponential convergence to the uniform distribution as $t\to\infty$. 
\vspace{0.2cm}

    \noindent (2) If $r_2(0)> 0$, then we will show that there exists a positive $\eta$ that satisfies the compatibility condition \eqref{compatibility} such that the limiting function for $f$ as $t \to \infty$ is the von-Mises distribution \eqref{von_Mises_eta_Phi}, instead of the uniform distribution $f_{\infty}^{\rm c}$. 
    Suppose that the limit distribution is uniform, and the associated limiting value for $r_2(t)$ as $t\to \infty$ should be zero. Using \eqref{d2r2}, we obtain
\begin{align*}
    \frac{d}{dt}r_2^2(t)=8\kappa r_2^2(t)\left[-\int_{\s^1}f(t,\theta)\cos 4(\theta-\phi_2(t))d\theta+\left(\frac{1}{4}-\frac{\sigma}{\kappa}\right)\right].
\end{align*}
Since our limit distribution is uniform, it is easy to see that 
\begin{equation*}
    \int_{\s^1}f(t,\theta)\cos 4(\theta-\phi_2(t))d\theta
\end{equation*}
is small enough when $t$ is large enough. Based on the case $\frac{\sigma}{\kappa}<\frac{1}{4}$, $r_2$ is increasing when $t$ is large enough, which gives a contradiction. Therefore, the steady states must be the nontrivial case. 

    Now we take the positive $\eta$ given by Theorem \ref{thm_unif_eta}, and define the von-Mises type function $M_{\eta,\phi_2(t)}$ where the angle $\phi_2(t)$ in order parameters is determined by the solution $f$.
    Assume that \eqref{converge} does not hold for some $m\in \mathbb{Z}_{0}$.
    Then, there exists $\varepsilon>0$ and a sequence $\{t_k\}_{k=1}^{\infty}$ such that
    \begin{equation}\label{new_1}
       \Vert f(t_k)-M_{\eta,\phi_2(t_k)}\Vert_{H_{\rm per}^m}\geq \varepsilon \quad \forall\   k\geq 1. 
    \end{equation}
    By the LaSalle invariance principle, there exists $\{\omega_k\}_{k=1}^{\infty}\in \s^1$ such that  
    \[
    \Vert f(t_k)-M_{\eta, \omega_k}\Vert_{H_{\per}^m}\to 0.
    \] 
    Since $\s^1$ is compact, there exists a subsequence of $\{\omega_{k}\}_{k=1}^{\infty}$ which converges to some $\Phi$. We still denote the sequence by $\{\omega_{k}\}_{k=1}^{\infty}$. Then we obtain the following convergence: 
    \[
    \Vert f(t_{n_k})-M_{\eta,\Phi}\Vert_{H_{\rm per}^m}\to 0 \quad \text{as}\ \  k\to\infty
    \]
    In particular, $\phi_2(t_k)\to \Phi$, resulting in the convergence of $M_{\eta,\phi_2(t_k)}$ to $M_{\eta,\Phi}$ in $H_{\per}^m$. Now we combine the results above and the Minkowski inequality to find
    \begin{equation*}
       \Vert f(t_k)-M_{\eta,\phi_2(t_k)}\Vert_{H_{\rm per}^m}\to 0 \quad \text{ as }\ k\to\infty.
    \end{equation*}
    This is contradictory to \eqref{new_1}.
\end{proof}

Now we study the asymptotic convergence behavior for the nontrivial case in the sense that
\[
f(t,\theta) = \big(1 + h(t,\theta)\big)M_{\eta,\phi_2(t)}(\theta),
\]
where $\eta>0$ is the same as in Proposition \ref{prop_supercritical} given by the compatibility condition, $\phi_2(t)$ is associated with $f$. The function $h$ is well-defined since $M_{\eta,\phi_2(t)}$ is positive, and the perturbation function $h\in H_{\rm per}^2$ always satisfies that
 $\|h(t)\|_{L^\infty_{\rm per}} < \varepsilon$ for a given $\varepsilon>0$.
 \begin{lemma}
	Let $f$ be the classical solution to the nonlinear Fokker-Planck equation \eqref{IBVP_hom}. Moreover, we assume that $r_2(t)>0$, for all $t\geq 0$. Then we have  
\begin{equation}\label{d1phin}
    \frac{d}{dt}\phi_2(t)=-\frac{\kappa}{r_2(t)}\int_{\s^1}f(t,\theta) L[f](t,\theta)
    \cos n(\theta-\phi_n)d\theta.
\end{equation}
\end{lemma}
\begin{proof}
	Note that
	\[
	r_2e^{\mathrm{i}n\phi_2}=\int_{ \s^1}f(t,\theta)e^{2\mathrm{i}\theta}d\theta.
	\]
	We take derivative on both sides to get
	\[
	\left(2\mathrm{i}r_2(t)\frac{d}{dt}\phi_2(t)+\frac{d}{dt}r_2(t)\right)e^{2\mathrm{i}\phi_2(t)}=\int_{\s^1}\frac{\pa}{\pa t}f(t,\theta)e^{2\mathrm{i}\theta}d\theta.
	\]
	Next, we multiply on both sides by $e^{-2\mathrm{i}\phi_2}$  and take the imaginary part to get
	\begin{align*}
		2r_2(t)\frac{d}{dt}\phi_2(t)=&\int_{\s^1}\frac{\pa}{\pa t}f(t,\theta)\sin 2(\theta-\phi_2(t))d\theta\\
		=&\int_{\s^1}\Big[-\kappa\pa_{\theta}(L[f](t,\theta)f(t,\theta))+\sigma\pa_{\theta\theta}f(t,\theta)\Big]\sin 2(\theta-\phi_2(t))d\theta\\=
		&\int_{\s^1}-2\kappa L[f](t,\theta)f(t,\theta)\cos 2(\theta-\phi_2(t))d\theta.
	\end{align*}
 Since $r_2(t) > 0$, we divide by $2r_2(t)$ on both sides to get the desired result \eqref{d1phin}.
\end{proof}
\begin{theorem}
    Suppose that $\Vert f(t)-M_{\eta,\phi_2(t)}\Vert_{H_{\per}^n}$ is uniformly bounded for $t\in [0,\infty)$ by a positive constant $K$, with $n$ large enough. Then, there exist $\Phi\in \s^1$ and $\delta, r, C>0$ such that if $\Vert f_0-M_{f_0}\Vert_{L^{2}_{\per}}<\delta$, we have 
    \begin{equation*}
        \Vert f(t)-M_{\eta, \Phi}\Vert_{L^2_{\per}}\leq 
 C \Vert f_0-M_{f_0}\Vert_{L^{2}_{\per}}e^{-rt}.
    \end{equation*}
\end{theorem}
\begin{proof}
Note that $\la h\rangle_{{M}_{\eta,\phi_2(t)}}=0$. 
Then we can do the expansion of the free energy and obtain the dispassion function around $M_{\eta,\phi_2(t)}$. 
Based on the definition \eqref{free_energy}, we have
\[
    \begin{split}
        &\mathcal{F}\left[f\right](t)-\mathcal{F}\left[M_{\eta,\phi_2(t)}\right](t)\\
        &\hspace{0.2cm}=\frac{\sigma}{2}\left\la h^2\right\ra_{M_{\eta,\phi_2(t)}}-\frac{\kappa}{4}\int_{\s^1\times\s^1}h(\theta)M_{\eta,\phi_2(t)}(\theta)h(\theta_*)M_{\eta,\phi_2(t)}(\theta_*)\cos2(\theta_*-\theta)d\theta d\theta_*\\
        &\hspace{0.5cm}+\mathcal{O}\left(\Vert h(t)\Vert_{L_{\per}^{\infty}}^3\right).
    \end{split}
\]
This implies
\begin{align*}
\mathcal{F}\left[f\right](t)=&\mathcal{F}\left[M_{\eta,\phi_2(t)}\right](t) + \frac{\sigma}{2} \left\langle h^2(t,\cdot)\right\rangle_{M_{\eta,\phi_2(t)}}\\
&-\frac{\kappa}{4}\left(\left\langle h(t,\cdot)\cos 2\cdot\right\rangle_{M_{\eta,\phi_2(t)}}^2
+ \left\langle h(t,\cdot)\sin 2\cdot\right\rangle_{M_{\eta,\phi_2(t)}}^2\right) 
+\mathcal{O}\left(\varepsilon^3\right) \\
=: &\mathcal{F}_{\varepsilon}\left[f\right](t) +\mathcal{O}\left(\varepsilon^3\right) .
\end{align*}
Note that
\begin{align*}
  &\left\langle h(t,\cdot)\cos 2\cdot\right\rangle_{M_{\eta,\phi_2(t)}}^2+ \left\langle h(t,\cdot)\sin 2\cdot\right\rangle_{M_{\eta,\phi_2(t)}}^2\\
&\hspace{0.5cm}=\left\langle h(t,\cdot)\cos 2\left(\cdot-\phi_2(t)\right)\right\rangle_{M_{\eta,\phi_2(t)}}^2
+ \left\langle h(t,\cdot)\sin 2\left(\cdot-\phi_2(t)\right)\right\rangle_{M_{\eta,\phi_2(t)}}^2.  
\end{align*}
Thanks to $\eqref{eqn_rn}_2$ with $n=2$, we have
\begin{align*}
\left\langle h(t,\cdot)\sin 2\left(\cdot-\phi_2(t)\right)\right\rangle_{M_{\eta,\phi_2(t)}}
&=\int_{\mathbb{S}^1} \left(f(t,\theta) - M_{\eta,\phi_2(t)}(\theta)\right)\sin 2\left(\theta-\phi_2(t)\right)\,d\theta = 0.
\end{align*}
Hence, we have
\begin{align}
\begin{aligned}\label{difference_F}
    &\mathcal{F}\left[f\right](t)-\mathcal{F}\left[M_{\eta,\phi_2(t)}\right](t) \\
&\hspace{0.5cm}=\mathcal{F}_{\varepsilon}\left[f\right](t)  -\mathcal{F}\left[M_{\eta,\phi_2(t)}\right](t) +\mathcal{O}\left(\varepsilon^3\right) \\
&\hspace{0.5cm}=\frac{\sigma}{2} \left\langle h^2(t,\cdot)\right\rangle_{M_{\eta,\phi_2(t)}}
-\frac{\kappa}{4}\left\langle h(t,\cdot)\cos 2\left(\cdot-\phi_2(t)\right)\right\rangle_{M_{\eta,\phi_2(t)}}^2
+\mathcal{O}\left(\varepsilon^3\right).
\end{aligned}
\end{align}
For the dissipation rate, we have 
\[
\mathcal{D}\left[M_{\eta,\phi_2(t)}\right](t) \equiv 0
\]
and
\begin{align*}
&\mathcal{D}[f](t)\\
=&\int_{\s^1} f(t,\theta)\left(\sigma \partial_\theta\log f(t,\theta)-\kappa L[f](t,\theta)\right)^2 d\theta\\
=&\left\la \big(1+h(t,\theta)\big)\left\vert \partial_\theta \left(\sigma \log \big(1+h(t,\theta)\big)-\frac{\kappa}{2}\left\langle h(t,\cdot) \cos 2(\cdot-\theta)\right\rangle_{M_{\eta,\phi_2(t)}}\right)\right\vert^2\right\ra_{M_{\eta,\phi_2(t)}} \\
\geq &\; C_{\eta}(1-\varepsilon)\left\langle \left(\varphi(t,\cdot) - \left\langle\varphi(t,\cdot)\right\rangle_{M_{\eta,\phi_2(t)}}\right)^2\right\rangle_{M_{\eta,\phi_2(t)}}\\
= &\; C_{\eta}(1-\varepsilon)\left(\left\langle \varphi^2(t,\cdot)\right\rangle_{M_{\eta,\phi_2(t)}} - \left\langle\varphi(t,\cdot)\right\rangle_{M_{\eta,\phi_2(t)}}^2\right).
\end{align*}
The last inequality used the weighted Poincar{\'e}'s inequality:
\[
\left\la \left\vert \partial_\theta \varphi\right\vert^2\right\ra_{M_{\eta,\phi_2(t)}}\geq C_{\eta} \left\la \left(\varphi-\la \varphi\ra_{M_{\eta,\phi_2(t)}}\right)^2\right\ra _{M_{\eta,\phi_2(t)}}\quad \forall \ \varphi \in H_{\rm per}^1.
\]
This can be deduced from the standard Poincar{\'e}'s inequality since $M_{\eta,\phi_2(t)}$ is positive and bounded.
Here, the positive constant $C_{\eta}$ depends on $\eta$ only. Now, we take
\[
\varphi(t,\theta) = \sigma \log \big(1+h(t,\theta)\big)-\frac{\kappa}{2}\left\langle h(t,\cdot) \cos 2(\cdot-\theta)\right\rangle_{M_{\eta,\phi_2(t)}}
\]
to derive the desired inequality. It follows that
\begin{align}\label{eq_D_eps}
\mathcal{D}[f](t) =: \mathcal{D}_{\varepsilon}[f](t) + \mathcal{O}\left(\varepsilon^3\right)
\geq C_{\eta}(1-\varepsilon){\rm I} + \mathcal{O}\left(\varepsilon^3\right),
\end{align}
where
\begin{align*}
{\rm I} = &\sigma^2\left\la h^2\right\ra_{M_{\eta,\phi_2(t)}}+\frac{\kappa^2}{4}\la h\cos2(\cdot-\phi_2(t))\ra_{M_{\eta,\phi_2(t)}}^2\\
    &\times\left(\left\la \cos^2 2(\cdot-\phi_2(t))\right\ra_{M_{\eta,\phi_2(t)}} - \la \cos2(\cdot-\phi_2(t))\ra_{M_{\eta,\phi_2(t)}}^2 - \frac{4\sigma}{\kappa}\right).
\end{align*}
Here we have used the relations: 
    \begin{equation}
        \la h\sin2(\theta-\phi_2(t))\ra_{M_{\eta,\phi_2(t)}}=\la (1+h)\sin2(\theta-\phi_2(t))\ra_{M_{\eta,\phi_2(t)}}=0.
    \end{equation}
Next, we claim that there exists a positive constant $\mu$ such that
\begin{equation}\label{new_2}
    \mathcal{D}[f](t)\geq 2\mu\left(\mathcal{F}[f](t)-\mathcal{F}[M_{\eta,\phi_2(t)}](t)\right).
\end{equation}
{\it Proof of Claim \eqref{new_2}}:
Based on previous calculation, we need to deal with $\left\la h^2\right\ra_{M_{\eta,\phi_2(t)}}$ more carefully. 
The idea is to split the term $\left\la h^2\right\ra_{M_{\eta,\phi_2(t)}}$ into $\la \cos2(\cdot-\phi_2(t))\ra_{M_{\eta,\phi_2(t)}}^2$ and other terms. For this, we introduce a new function $\tilde h=\tilde h(t,\theta)$ such that
\begin{align*}
    h(t,\theta)=&\tilde h(t,\theta)+\left(\cos 2(\theta-\phi_2(t))-\la\cos2(\theta-\phi_2(t))\ra_{M_{\eta,\phi_2(t)}}\right)\alpha(t),
\end{align*}
where
\begin{equation*}
    \alpha(t)=\frac{\la h\cos2(\cdot-\phi_2(t))\ra_{M_{\eta,\phi_2(t)}}}{\left\la \left(\cos 2(\cdot-\phi_2(t))-\left\la \cos 2(\cdot-\phi_2(t))\right\ra_ {M_{\eta,\phi_2(t)}}\right)^2\right\ra_{\M}}.
\end{equation*}
For the denominator in $\alpha(t)$, we have
\begin{align*}
&\left\la \left(\cos 2(\cdot-\phi_2(t))-\left\la \cos 2(\cdot-\phi_2(t))\right\ra_ {M_{\eta,\phi_2(t)}}\right)^2\right\ra_{\M}\\
&\hspace{1cm}=\left\la \cos^2 2(\cdot-\phi_2(t))\right\ra_{\M} - \left\la \cos 2(\cdot-\phi_2(t))\right\ra_{\M}^2\\
&\hspace{1cm}= 1- \frac{2\sigma}{\kappa} - \left(\frac{2\sigma}{\kappa}\eta\right)^2 =:\varkappa \in \left(0, \frac{2\sigma}{\kappa}\right).
\end{align*}
Note that
\begin{equation*}
    \left\la \tilde h\right\ra_{M_{\eta,\phi_2(t)}}=\left\la \tilde h\cos2(\cdot-\phi_2(t))\right\ra_{M_{\eta,\phi_2(t)}}=0.
\end{equation*}
Hence, we have
\begin{align}
\begin{aligned}\label{h^2_weighted_average}
    \left\la h^2\right\ra_{M_{\eta,\phi_2(t)}}
    &=\left[\la \cos^2 2(\cdot-\phi_2)\ra_{\M}-\left(\frac{2\sigma}{\kappa}\eta\right)^2\right]\alpha^2(t)+\left\la\tilde h^2\right\ra_{\M}\\
    &= \varkappa\alpha^2(t)+\left\la\tilde h^2\right\ra_{\M}.
\end{aligned}
\end{align}
We substitute \eqref{h^2_weighted_average} into \eqref{difference_F} and \eqref{eq_D_eps} to obtain 
\begin{align}\label{difference_F2}
    &\mathcal{F}_{\varepsilon}[f](t)-\mathcal{F}[M_{\eta,\phi_2(t)}](t)\notag\\
    &\hspace{1cm}=\frac{\sigma}{2}\left(\varkappa\alpha^2(t)+\left\la\tilde h^2\right\ra_{\M}\right)-\frac{\kappa}{4}\la h\cos2(\cdot-\phi_2)\ra_{M_{\eta,\phi_2(t)}}^2\notag
    \\
    &\hspace{1cm}=\frac{\kappa}{4}\varkappa\left(\frac{2\sigma}{\kappa}-\varkappa\right)\alpha^2(t)
    +\frac{\sigma}{2}\left\la\tilde h^2\right\ra_{\M},
\end{align}
and
\begin{align*}
    &\mathcal{D}_{\varepsilon}[f](t)\\
    \geq&\Lambda \left(1-\varepsilon\right)\left[\sigma^2\left(\varkappa\alpha^2(t)+\left\la\tilde h^2\right\ra_{\M}\right)
    +\frac{\kappa^2}{4}\left(\varkappa - \frac{4\sigma}{\kappa}\right)\left\la h\cos2(\cdot-\phi_2(t))\right\ra_{M_{\eta,\phi_2(t)}}^2 \right]\nonumber\\
    =&\Lambda \left(1-\varepsilon\right)\left[\frac{\kappa^2}{4}\varkappa\left(\frac{2\sigma}{\kappa} - \varkappa\right)^2\alpha^2(t)
    +\sigma^2\left\la\tilde h^2\right\ra_{\M}\right] \\
    =&\Lambda \left(1-\varepsilon\right)\kappa\left(\frac{2\sigma}{\kappa} - \varkappa\right)\left[\frac{\kappa}{4}\varkappa\left(\frac{2\sigma}{\kappa} - \varkappa\right)\alpha^2(t)
    +\frac{\sigma^2}{\kappa\left(\frac{2\sigma}{\kappa} - \varkappa\right)}\left\la\tilde h^2\right\ra_{\M}\right].
\end{align*}
Thanks to Lemma \ref{lemma_kappa_sigma}, we have
\[
0<\frac{2\sigma}{\kappa}-\varkappa< \frac{2\sigma}{\kappa} 
\qquad \Longrightarrow \qquad \frac{\sigma^2}{\kappa\left(\frac{2\sigma}{\kappa} - \varkappa\right)} > \frac{\sigma}{2}.
\]
It follows that
\[
\mathcal{D}_{\varepsilon}[f](t)\geq 2\mu\left(\mathcal{F}_{\varepsilon}[f](t)-\mathcal{F}[M_{\eta,\phi_2(t)}](t)\right)
\text{ with }
\mu=\Lambda\left(1-\varepsilon\right)\frac{\kappa}{2}\left(\frac{2\sigma}{\kappa}-\varkappa\right).
\]
Next we neglect higher order terms and use Gronwall's lemma to derive 
\begin{equation}\label{bound_1}
    \mathcal{F}[f](t)-\mathcal{F}[M_{\eta,\phi_2(t)}](t)\leq e^{-2\mu t}\left(\mathcal{F}[f](0)-\mathcal{F}[M_{\eta,\phi_2(0)}](0)\right).
\end{equation}
By \eqref{h^2_weighted_average} and \eqref{difference_F2}, there exist positive constants $c_1$ and $c_2$ such that 
\begin{equation}\label{bound_2}
    c_1\left\la h^2\right\ra_{M_{\eta,\phi_2(t)}}\leq \mathcal{F}[f](t)-\mathcal{F}[M_{\eta,\phi_2(t)}](t)\leq c_2\left\la h^2\right\ra_{M_{\eta,\phi_2(t)}}.
\end{equation}
Since $M_{\eta,\phi_2(t)}$ is bounded from above and below, there exist positive constants $c_3$ and $c_4$ such that
\begin{equation}\label{bound_3}
    c_3\Vert f(t)-M_{\eta, \phi_2(t)}\Vert_{L_{\per}^2}^2\leq \left\la h^2\right\ra_{M_{\eta,\phi_2(t)}}\leq c_4\Vert f(t)-M_{\eta, \phi_2(t)}\Vert_{L_{\per}^2}^2.
\end{equation}
We combine \eqref{bound_1} - \eqref{bound_3} to find that there exists constant $C_1>0$ such that
\begin{equation*}
    \Vert f(t)-M_{\eta, \phi_2(t)}\Vert_{L_{\per}^2}\leq C_1e^{-\mu t}\Vert f_0-M_{\eta, \phi_2(0)}\Vert_{L_{\per}^2}.
\end{equation*}
Next we study the evolution of $\phi_2(t)$.
\begin{align*}
    \left|\frac{d}{dt}\phi_2(t)\right|&=\kappa\left|\int_{\s^1}f(t,\theta)\cos2(\theta-\phi_2(t))\sin2(\theta-\phi_2(t))d\theta\right|\\
    &=\kappa\left|\int_{\s^1}\left(f(t,\theta)-M_{\eta, \phi_2(t)}\right)\cos2(\theta-\phi_2(t))\sin2(\theta-\phi_2(t))d\theta\right|\\
    &\leq \sqrt{2\pi}\kappa \Vert f(t)-M_{\eta, \phi_2(t)}\Vert_{L_{\per}^2} \leq \sqrt{2\pi}\kappa C_1e^{-\mu t}\Vert f_0-M_{\eta, \phi_2(0)}\Vert_{L_{\per}^2}.
\end{align*}
Therefore, $\phi_2$ converges to some $\Phi$ as $t\to \infty$, and 
\begin{align*}
    \begin{aligned}
       \left\vert \Phi-\phi_2(t)\right\vert&\leq \int_{t}^{\infty}\left\vert\frac{d}{dt}\phi_2(\tau)\right\vert d\tau\leq \sqrt{2\pi}\kappa C_1\int_{t}^{\infty}e^{-\mu \tau}d\tau\Vert f_0-M_{\eta, \phi_2(0)}\Vert_{L_{\per}^2}\\&= C_2e^{-\mu t} \Vert f_0-M_{\eta, \phi_2(0)}\Vert_{L_{\per}^2},  
    \end{aligned}
\end{align*}
where $C_2 = \frac{\sqrt{2\pi}\kappa C_1}{\mu}$.
Since $M_{\eta, \phi_2}$ is Lipschitz continuous with respect to $\phi_2$, we get the desired estimate: 
\begin{align*}
    \Vert f(t)-M_{\eta, \Phi}\Vert_{L_{\per}^2}&\leq \Vert f(t)-M_{\eta, \phi_2(t)}\Vert_{L_{\per}^2}+\Vert M_{\eta, \Phi}-M_{\eta, \phi_2(t)}\Vert_{L_{\per}^2}\\
    &\leq \Vert f(t)-M_{\eta, \phi_2(t)}\Vert_{L_{\per}^2}+C_3\vert \Phi-\phi_2(t)\vert\\
    &\leq C_4\Vert f(t)-M_{\eta, \phi_2(t)}\Vert_{L_{\per}^2}\leq Ce^{-\mu t}\Vert f_0-M_{\eta, \phi_2(0)}\Vert_{L_{\per}^2}.
\end{align*}
\end{proof}
\begin{remark}
    From \eqref{difference_F2}, we obtain that for $0<\frac{\sigma}{\kappa}<\frac{1}{4}$, 
    \[
    \mathcal F[f^c_{\infty}]>\mathcal F[M_{\eta}].
    \]
    Thus, uniform distribution is not stable in the case of $0<\frac{\sigma}{\kappa}<\frac{1}{4}$. 
\end{remark}
\section{Phase transition for nonconstant noise} In this section, we study phase transition in nonconstant noise setting. 

\vspace{0.15cm}
Consider a particular noise term given by  
\[
D[f](t,\theta)=\frac{1}{4}\left(\left\vert \int_{\s^1}f(t,\theta_*)e^{2i\theta_*}d\theta_*\right\vert+\int_{\s^1}\cos2(\theta_*-\theta)f(t,\theta_*)d\theta_*\right)^2.
\]
The  IBVP for \eqref{IBVP} in spatially homogeneous setting becomes
\begin{equation}\label{IBVP_hom_degenerate}
    \left\{\begin{aligned}
       &\pa_t f+\kappa\pa_\theta\big(L[f]f\big)=\sigma\pa_\theta (D[f]\pa_{\theta}f), \quad \theta \in \s^1, \; t >0,\\
&f|_{t=0}=f_0, \quad f(t,\theta) = f(t,\theta+2\pi).
    \end{aligned}\right.
\end{equation}
\begin{remark}
    If $f$ is a probability density function for every $t\geq 0$, then $D[f]$ can be expressed in terms of $r_2$ and $\phi_2$:
    \begin{equation}\label{equ_D_nonconstant}
        D[f]=r_2^2\cos^4(\theta-\phi_2).
    \end{equation}
    The corresponding diffusion term in the particle system in \eqref{A3} reads
    \[
    D_t^j(X_t,\Theta_t)=\frac{1}{2}\left[\frac{1}{N}\sum_{k=1}^N\cos\left(2(\theta_t^k-\theta_t^j)\right)+\left|\frac{1}{N}\sum_{k=1}^N\exp{\left(2\mathrm{i}\left(\theta_t^k-\theta_t^j\right)\right)}\right|\right].
    \]
\end{remark}
For this nonconstant case, the diffusion term is degenerate. Therefore, the standard parabolic PDE theory cannot be applied directly to obtain a local well-posedness. In this sequel, we mainly focus on equilibrium.  
Now we introduce the function $g=g(t,\theta)$: 
\begin{equation}\label{def_g}
	g(t,\theta)=f(t,\theta)+f(t,\theta+\pi), \qquad  \theta \in [0,\pi).
\end{equation}
We will see the advantage of studying $g$ instead of $f$ later. It is easy to see that $g$ has a period of $\pi$ and 
\begin{align*}
	\int_0^{\pi}g(t,\theta)d\theta=
	&\int_0^{\pi}\Big(f(t,\theta)+f(t,\theta+\pi)\Big)d\theta=
	\int_0^{2\pi}f(t,\theta)d\theta,
\end{align*}
i.e. $g$ is a probability density function on $[0,\pi)$ for $f$ being a probability density function on $\s^1$. 
Based on the mean-field equation for $f$, we derive the Fokker-Planck equation for $g$. 

\begin{lemma}
Let f be a classical solution to \eqref{IBVP_hom_degenerate}. Then, the function $g$ defined by \eqref{def_g} satisfies the following Fokker-Planck equation:
\begin{equation}\label{F-P_double_angle}
	\partial_t g=-\kappa\partial_\theta\left(\tilde L[g]g\right)+\sigma\partial_\theta\left(\tilde D[g]\partial_\theta g\right),
\end{equation}
where the nonlocal terms are given as follows:
\begin{align*}
    &\tilde L[g](t,\theta)=\int_0^{\pi}\sin 2(\theta_*-\theta)g(t,\theta_*)d\theta_*,\\
&\tilde D[g](t,\theta)=\frac{1}{4}\left(\left\vert \int_0^{\pi}g(t,\theta_*)e^{2i\theta_*}d\theta_*\right\vert+\int_0^{\pi}\cos2(\theta_*-\theta)g(t,\theta_*)d\theta_*\right)^2.
\end{align*}
\end{lemma}
\begin{proof}
Note that $L[f](t,\theta) = L[f](t,\theta+\pi)$ and
	\begin{align*}
		L[f](t,\theta)=\int_{\s^1}\sin 2(\theta_*-\theta)f(t,\theta_*)d\theta_*
  =\int_0^{\pi}\sin 2(\theta_*-\theta)g(t,\theta_*)d\theta_*.
	\end{align*}
That is to say, $L[f]$ can be regarded as a functional of $g$, and we denote it by $\tilde{L}[g]$. Meanwhile, $D[f](t,\theta) = D[f](t,\theta+\pi)$, 
\begin{align*}
&\left| \int_{\s^1}f(t,\theta_*)e^{2\mathrm{i}\theta_*}\,d\theta_*\right|=\left| \int_{0}^{\pi}f(t,\theta_*)e^{2\mathrm{i}\theta_*}\,d\theta_*
+\int_{\pi}^{2\pi}f(t,\theta_*)e^{2\mathrm{i}\theta_*}\,d\theta_*\right|\\
  &\hspace{1cm}=\left| \int_{0}^{\pi}\left( f(t,\theta_*) + f(t,\theta_*+\pi)\right)e^{2\mathrm{i}\theta_*}\,d\theta_*\right|=\left|\int_0^{\pi}g(t,\theta_*)e^{2\mathrm{i}\theta_*}e^{-2\mathrm{i}\theta}\,d\theta_*\right|,
\end{align*}
and
\begin{align*}	
		\int_{\s^1}\cos 2(\theta_*-\theta)f(t,\theta_*)\,d\theta_*
  =\int_0^{\pi}\cos 2(\theta_*-\theta)g(t,\theta_*)\,d\theta_*.
	\end{align*}
Therefore, $D[f]$ can be rewritten as a functional of $g$, denoted by $\tilde{D}[g]$. 
Note that
	\begin{align*}
		\partial_t g(t,\theta)=&\,\partial_t\Big(f(t,\theta)+f(t,\theta+\pi)\Big)\\
  =&-\kappa\partial_\theta\Big(L[f](t,\theta)f(t,\theta)\Big)+\sigma\partial_\theta\Big(D[f](t,\theta)\partial_\theta f(t,\theta)\Big)\\
  &-\kappa\partial_\theta\Big(L[f](t,\theta+\pi)f(t,\theta+\pi)\Big)+\sigma\partial_\theta\Big(D[f](t,\theta+\pi)\partial_\theta f(t,\theta+\pi)\Big)\\
  =&-\kappa\partial_\theta\Big(\tilde L[g](t,\theta)g(t,\theta)\Big)+\sigma\partial_\theta\Big(\tilde D[g](t,\theta)\partial_\theta g(t,\theta)\Big).
	\end{align*}
\end{proof}

Next, we introduce an operator $Q$ and rewrite \eqref{F-P_double_angle} as
\begin{equation}\label{F-P_double_angle_Q}
	\partial_t g=Q[g]:=\sigma\partial_\theta\left(-\frac{\kappa}{\sigma}\tilde L[g]g+\tilde D[g]\partial_\theta g\right).
\end{equation}
Since the evolution of $g$ is self-consistent, we can study the steady state for $g$. First, we define the equilibrium state for \eqref{F-P_double_angle_Q}.
\begin{definition}
	$\tilde{M}=\tilde{M}(\theta)$ is an equilibrium state for \eqref{F-P_double_angle_Q},
    if the following conditions hold:
	\begin{itemize}
		\item[(i)] $\tilde{M}(\theta)\in C^2(\R)$, is nonnegative and has a period of $\pi$;
        \vspace{0.1cm}
		\item[(ii)] $\int_0^{\pi}\tilde{M}(\theta)\,d\theta = 1$ $\quad$ and $\quad$ $Q\left[\tilde{M}\right]=0$.
\end{itemize}
\end{definition}
To study the equilibrium state for \eqref{F-P_double_angle_Q}, we consider order parameters with respect to $g$: 
\begin{equation*}
	 r_2[g](t)e^{2\mathrm{i}\phi_2[g](t)}=\int_0^\pi g(t,\theta)e^{2\mathrm{i} \theta}d\theta,
\end{equation*}
where
\begin{equation*}
	r_2[g](t)=\left\vert\int_0^\pi g(t,\theta)e^{2\mathrm{i} \theta}d\theta\right\vert.
\end{equation*}
Here we still use the notations $r_2$ and $\phi_2$, since they have the same value for $f$ and its corresponding $g$ in \eqref{def_g}. To show the difference, we add a square bracket of $[\cdot]$ to indicate the dependence of $f$ or $g$.
\begin{proposition}\label{equilibrium_g}
    For system \eqref{F-P_double_angle_Q}, there are only two kinds of equilibria. One is all $C^2$ probability distribution function $\tilde M$ such that $r_2\left[\tilde{M}\right]=0$, and the other is the generalized von-Mises distribution:
    \begin{equation}\label{equilibrium_nonconstant}
        \tilde M_{\tilde\eta,\tilde\Phi}(\theta)=\frac{1}{Z_{\tilde\eta}}\exp{\left(-\frac{\tilde\eta}{\cos^2(\tilde\Phi-\theta)}\right)},
    \end{equation}
    where $\tilde \eta>0$, and $Z_{\tilde\eta}$ is a normalization constant.
\end{proposition}
\begin{proof}
We consider the following cases:
\[
\text{Either} \quad r_2\left[\tilde M\right]=0 \quad \text{or}\quad r_2\left[\tilde M\right]\neq0.
\]
$\bullet$ Case A: Consider the case $r_2\left[\tilde M\right]=0$, then we have
\[
\tilde L\left[\tilde M\right](\theta)=\tilde D\left[\tilde M\right](\theta)=0.
\] 
Therefore, every $C^2$ probability distribution function $\tilde M$ such that $r_2\left[\tilde M\right]=0$ is an equilibrium.

\vspace{0.1cm}
\noindent$\bullet$ Case B: Consider the case $r_2\left[\tilde M\right]\neq 0$.
Note that
\[
Q\left[\tilde{M}\right](\theta) = \sigma \pa_\theta\left\{\tilde M(\theta)\tilde D\left[\tilde M\right](\theta)\left[\partial_{\theta}\log \tilde M-\frac{\kappa}{\sigma}\frac{\tilde L\left[\tilde M\right]}{\tilde D\left[\tilde M\right]}\right]\right\}.
\]
Therefore, we only need to verify
\begin{equation}\label{eqn_partiallogM}
\partial_{\theta}\log \tilde M(\theta)-\frac{\kappa}{\sigma}\frac{\tilde L\left[\tilde M\right](\theta)}{\tilde D\left[\tilde M\right](\theta)} = 0 \quad\forall \ \theta\in\left(-\frac{\pi}{2},\frac{\pi}{2}\right).
\end{equation}
Indeed, since we focus on the case $r_2>0$, and by the definition of $\tilde M\in C^2$, we know that $\tilde D$ can be written in the following form of \eqref{equ_D_nonconstant}. Then we combine \eqref{equ_D_nonconstant} with \eqref{eqn_partiallogM} to get the characterization of equilibrium in \eqref{equilibrium_nonconstant}.

\vspace{0.15cm}
    Now we turn to the proof of \eqref{eqn_partiallogM}. Since $\tilde M$ is an equilibrium, $Q\left[\tilde M\right](\theta)=0$ for each $\theta\in \left(-\frac{\pi}{2},\frac{\pi}{2}\right)$. 
    Without loss of generality, we set $\phi_2\left[\tilde M\right]=0$, and we set
    \[
    \psi(\theta):=\log \tilde M(\theta)-\frac{\kappa}{\sigma}\displaystyle\int_{0}^{\theta}\frac{\tilde L\left[\tilde M\right](\theta_*)}{\tilde D\left[\tilde M\right](\theta_*)}d\theta_*,\quad \theta\in\left(-\frac{\pi}{2},\frac{\pi}{2}\right).
    \]
    We multiply $Q\left[\tilde M\right]$ by $\sigma \psi(\theta)$ and use the integration by parts over $(-\frac{\pi}{2},\frac{\pi}{2})$ to obtain 
    \begin{equation}\label{eqn_tildeMtildeD}
        \int_{(-\frac{\pi}{2},\frac{\pi}{2})}\tilde M(\theta)\tilde D\left[\tilde M\right](\theta)\left\vert\sigma\pa_{\theta}\log\tilde M(\theta)-\kappa\frac{\tilde L\left[\tilde M\right](\theta)}{\tilde D\left[\tilde M\right](\theta)}\right\vert ^2 d\theta=0.
    \end{equation}
    Now we focus on the interval $(-\frac{\pi}{2},\frac{\pi}{2})$ since 
    \[
    \tilde D\left[\tilde M\right]\left(\frac{\pi}{2}\right)=0.
    \]
    We want to show that $\psi$ remains a constant on $(-\frac{\pi}{2},\frac{\pi}{2})$. Since $\tilde M$ is a  $C^2$ probability distribution function, $\left\{\theta|\ \tilde M(\theta)>0\right\}$ is relatively open to $(-\frac{\pi}{2},\frac{\pi}{2})$. Note that $\tilde D\left[\tilde M\right]>0$ on $(-\frac{\pi}{2},\frac{\pi}{2})$, so that $\left\{\theta|\ \tilde M(\theta)\tilde D\left[\tilde M\right](\theta)>0\right\}$ is relatively open to $(-\frac{\pi}{2},\frac{\pi}{2})$. In other words, $\left\{\theta|\ \tilde M(\theta)\tilde D\left[\tilde M\right](\theta)>0\right\}$ consists of finite relatively open sets. By \eqref{eqn_tildeMtildeD}, $\psi(\theta)$ remains a constant in every connected component where $\tilde M\tilde D\left[\tilde M\right]$ takes positive values. Therefore, $\psi^{-1}(C)$ is open for every $C\in \mathbb{R}$. 

    \vspace{0.1cm}
    Now if $\psi(\theta_k)=C$, $k=1,2, \cdots$, with $\theta_k$ converging to some $\theta_{\infty}\in(-\frac{\pi}{2},\frac{\pi}{2})$, then we have 
    \[
    \tilde M(\theta_k)=\exp\left(-\frac{\tilde \eta}{\cos^2(\theta_k)}+{C}\right),
    \]
    where $\tilde \eta$ is a positive constant. By the continuity of $\tilde M$, we can pass to the limit of $k\to \infty$ to get
    \[
    \tilde M(\theta_{\infty})=\exp\left(-\frac{\tilde \eta}{\cos^2(\theta_{\infty})}+C\right).
    \]
    Therefore, $\psi^{-1}(C)$ is relatively closed with respect to $(-\frac{\pi}{2},\frac{\pi}{2})$. Since $\psi^{-1}(C)$ is also relatively open with respect to $(-\frac{\pi}{2},\frac{\pi}{2})$, $\psi^{-1}(C)$ can only be $(-\frac{\pi}{2},\frac{\pi}{2})$ if it is not empty. Note that there exists $\theta_0\in(-\frac{\pi}{2},\frac{\pi}{2})$ such that $\tilde M(\theta_0)>0$. We take $C_0=\psi(\theta_0)$ and consider the set $\psi^{-1}(C_0)$, which is nonempty, relative closed and open to $(-\frac{\pi}{2},\frac{\pi}{2})$. As a consequence, $\psi$ takes the same value $C_0$ on $(-\frac{\pi}{2},\frac{\pi}{2})$ and the relation \eqref{eqn_partiallogM} holds.
\end{proof}
Based on the characterization of equilibrium, we focus on the generalized von-Mises case. Similar to the constant noise, we still want to know what the corresponding compatibility condition is, and how many solutions to the compatibility condition exist in different cases. 
\begin{proposition}\label{prop_unique_tilde_eta}
Suppose that $\tilde M$ is an equilibrium for \eqref{F-P_double_angle} with boundary condition of period $\pi$ with $r_2\left[\tilde M\right]>0$. Then, for all $\sigma,\kappa>0$, there exists a unique $\tilde \eta>0$ such that     \[
    \tilde M=\frac{1}{\tilde Z}\exp{\left(-\frac{\tilde \eta}{\cos^2\left(\tilde \Phi-\theta\right)}\right)},
    \]
    where $\phi_2\left[\tilde M\right]=\tilde \Phi$.
    In other words, for each $\sigma,\kappa>0$, generalized von-Mises distribution always exists and is unique. 
\end{proposition}
\begin{proof}
    We first derive the compatibility condition. Similar to the constant case, the equation reads
    \begin{equation*}
        \int_0^{\pi}\exp{\left(-\frac{\kappa}{\sigma r_2\cos^2(\Phi-\theta)}\right)}\cos 2(\theta-\Phi)d\theta=r_2\int_0^{\pi}\exp{\left(-\frac{\kappa}{\sigma r_2\cos^2(\Phi-\theta)}\right)}d\theta.
    \end{equation*}
    After translation, the above condition by setting $\Phi=0$ can be rewritten as 
    \begin{equation}\label{equ_com_con_degenerate}
        \int_{-\frac{\pi}{2}}^{\frac{\pi}{2}}\exp{\left(-\frac{\kappa}{\sigma r_2\cos^2\theta}\right)}\cos 2\theta d\theta=r_2\int_{-\frac{\pi}{2}}^{\frac{\pi}{2}}\exp{\left(-\frac{\kappa}{\sigma r_2\cos^2\theta}\right)}d\theta.
    \end{equation}
    We set $\tilde \eta=\frac{\kappa}{\sigma r_2}$. Then the condition becomes 
    \begin{equation}\label{compatibility_degenerate}
        \frac{\kappa}{\sigma}=\frac{\tilde \eta\int_{-\frac{\pi}{2}}^{\frac{\pi}{2}}\exp{\left(-\frac{\tilde \eta}{\cos^2\theta}\right)}\cos 2\theta d\theta}{\int_{-\frac{\pi}{2}}^{\frac{\pi}{2}}\exp{\left(-\frac{\tilde \tilde \eta}{\cos^2\theta}\right)}d\theta}.
    \end{equation}
    We define the function 
    \begin{equation*}
        \tilde\beta(\tilde \eta):=\frac{\tilde \eta\int_{-\frac{\pi}{2}}^{\frac{\pi}{2}}\exp{\left(-\frac{\tilde \eta}{\cos^2\theta}\right)}\cos 2\theta d\theta}{\int_{-\frac{\pi}{2}}^{\frac{\pi}{2}}\exp{\left(-\frac{\tilde \eta}{\cos^2\theta}\right)}d\theta},\ \ \tilde\eta\geq 0.
    \end{equation*}
    Then it is easy to see
    \[
    \tilde\beta(0)=0.
    \]
    Next, we study the derivative $\tilde\beta'(\tilde\eta)$ to obtain the correspondence between $\tilde\eta$ and $\frac{\kappa}{\sigma}$:
    \begin{align*}
        \tilde\beta'(\tilde\eta)&=\left\la \cos 2\theta -\tilde\eta\frac{\cos 2\theta}{\cos^2\theta}\right\ra_{\tilde M_{\tilde\eta}}+\tilde\eta\la \cos2\theta\ra_{\tilde M_{\tilde\eta}}\left\la \frac{1}{\cos^2\theta}\right\ra_{\tilde M_{\tilde\eta}}\nonumber\\
        &=\la \cos 2\theta\ra_{\tilde M_ {\tilde\eta}}+2\tilde\eta\left(\la \cos^2\theta\ra_{\tilde M_{\tilde\eta}}\left\la \frac{1}{\cos^2\theta}\right\ra_{\tilde M_{\tilde\eta}}-1\right)>0.
    \end{align*}
    Therefore, $\tilde\beta(\tilde\eta)$ is strictly increasing. Furthermore,
    \begin{equation*}
        \la \cos 2\theta\ra_{\tilde M_{\tilde\eta}}\to 1,\quad \text{as} \ \ \tilde\eta\to \infty.
    \end{equation*}
    This means  
    \[
    \{\gamma|\ \gamma=\tilde\beta(\tilde\eta), \tilde\eta\in \mathbb{R}_+\}=\mathbb{R}_+ .
    \]
     In other words, whatever values $\sigma$ and $\kappa$ take, there exists $\tilde\eta>0$ such that \eqref{compatibility_degenerate} holds. 
\end{proof}

After finding the equilibrium for \eqref{F-P_double_angle}, we return to the equilibrium cases for $f$ in the proposition below.

\begin{proposition}
  Suppose that $M$ is an equilibrium for system \eqref{IBVP_hom_degenerate}. Then, there are only two cases. One is that $r_2[M]=0$, while the other is
  \begin{equation*}
      M(\theta)=c_1\exp{\left(-\frac{\eta}{\cos^2(\Phi-\theta)}\right)}\Ind{\vert\theta-\Phi\vert<\frac{\pi}{2}}+c_2\exp{\left(-\frac{\eta}{\cos^2(\Phi-\theta)}\right)}\Ind{\frac{\pi}{2}<\vert\theta-\Phi\vert\leq\frac{3\pi}{2}},
  \end{equation*}
  where $c_1$  and $c_2$ are two nonnegative normalization constants, and they can be different.
\end{proposition}
\begin{proof}
    Since that $M$ is an equilibrium for system \eqref{IBVP_hom_degenerate}, then $\tilde M(t,\cdot):=M(t,\cdot)+M(t,\cdot+\pi)$ must be an equilibrium for \eqref{F-P_double_angle} with boundary condition of period $\pi$. Similar to the proof in Proposition \ref{equilibrium_g}, we consider two cases:

    \vspace{0.15cm}
    \noindent$\bullet$ Case A: If $r_2[M]=0$, then
         \[
         Q[M]=0.
         \]
    Therefore, every $C^2$ probability distribution function $\M$ such that $r_2[M]=0$ is an equilibrium.
    
         \vspace{0.15cm}
         \noindent$\bullet$ Case B:
    If $r_2[M]\neq 0$, then
    \[
    r_2\left[\tilde M\right]\neq 0.
    \]
      By Proposition \ref{equilibrium_g}, since $\tilde M$ is an equilibrium for \eqref{F-P_double_angle_Q}, one has
    \[
    \tilde M=\frac{1}{Z}\exp{\left(-\frac{\tilde\eta}{\cos^2(\tilde\Phi-\theta)}\right)}.
    \]
    We only need to get $M$ from $\tilde M$. Since $L[M]$ and $D[M]$ can be also expressed into functionals of $\tilde M$, they are both fixed whatever value $M$ takes as long as $\tilde M$ is fixed. That is to say, $\sigma \log M(\theta)-\int_{\phi_2[M]}^{\theta}\kappa\frac{ L[M](\theta_*)}{ D[M](\theta_*)}d\theta_*$ remains constant in every connected component where $M D[M]$ takes positive values. Similar to the proof in Proposition \ref{equilibrium_g}, $M$ must take the form of $c\exp{\left(-\frac{\eta}{\cos^2(\Phi-\theta)}\right)}$ in two connected components $(-\frac{\pi}{2}+\Phi,\frac{\pi}{2}+\Phi)$ and $(\frac{\pi}{2}+\Phi,\frac{3\pi}{2}+\Phi)$, where $\eta=\tilde \eta$ and $\Phi=\tilde \Phi$. Based on the above discussions, we get the desired form of $M$.
\end{proof}

For the discussions below, we only consider the cases near the equilibrium, that is to say, we fix $\phi_2$ to be a constant independent of $t$. Under this assumption, $\tilde M$ only depends on $\phi_2$ and $\theta$. Here we assume the well-posedness of the equation, that is to say, given a smooth probability distribution function $f_0$ on $[0,\pi)$, the solution $f$ of \eqref{IBVP_hom_degenerate} is unique and smooth with respect to $t$ and $\theta$. 

Since we have already known that there is only one solution to the compatibility condition \eqref{equ_com_con_degenerate}, there is no doubt that if the solution $g$ converges to some $\tilde M$, $\tilde M$ must be the unique equilibrium given by Proposition \ref{prop_unique_tilde_eta}. Thus we study the convergence of $g$ instead of $f$. In this part, we assume the convergence in the sense of $L^{\infty}_{\per}$ norm. Similar degenerate problems have been studied in \cite{alikakos1988stabilization,carrillo1998exponential,huang2016steady}. 

Here, we use the relative entropy to analyze the convergence to equilibrium. 
We define the relative entropy of the system as
\begin{equation}\label{def_relative_entropy_degenerate}
	\mathcal{H}\left[g|\tilde{M}\right](t):=\int_0^{\pi}g(t,\theta)\log\dfrac{g(t,\theta)}{\tilde{M}(\theta)}d\theta.
    \end{equation}
\begin{proposition}
Let g be the solution to \eqref{F-P_double_angle}. Then, the relative entropy $\mathcal{H}\left[g|\tilde{M}\right]$ is non-decreasing with respect to time $t$. 
\end{proposition}
\begin{proof}
    We differentiate \eqref{def_relative_entropy_degenerate} with respect to $t$ to obtain 
\begin{align}\label{new_13}
	&\frac{d}{dt}\mathcal{H}\left[g|\tilde{M}\right](t)=\frac{d}{dt}\int_0^{\pi}g\log\dfrac{g}{\tilde{M}}d\theta\notag\\
    &\hspace{1cm
    }=\int_0^{\pi}\left[\partial_t g\log\dfrac{g}{\tilde{M}}+g\partial_t\log\dfrac{g}{\tilde{M}}\right]d\theta=:\mathcal{I}_{31}+\mathcal{I}_{32},
\end{align}
Now, we estimate two terms in the above relation one by one.

\vspace{0.15cm}
\noindent$\bullet$ Case A (Estimate of $\mathcal{I}_{31}$): We use \eqref{F-P_double_angle_Q} to find 
\begin{align}\label{new_11}
	&\int_0^{\pi}\partial_t g\log\dfrac{g}{\tilde{M}}d\theta=\sigma\int_0^{\pi}\partial_\theta\left[\tilde D[g]M_g\partial_\theta\left(\dfrac{g}{\tilde{M}}\right)\right]\log\dfrac{g}{\tilde{M}}d\theta\notag\\
	&\hspace{1cm}=-\sigma\int_0^{\pi}\tilde D[g]\tilde{M}\partial_\theta\left(\dfrac{g}{\tilde{M}}\right)\partial_\theta\log\dfrac{g}{\tilde{M}}d\theta\notag\\
    &\hspace{1cm}=-\sigma\int_0^{\pi}\tilde D[g]g\left(\partial_\theta\log\dfrac{g}{\tilde{M}}\right)^2d\theta.
\end{align}

\vspace{0.15cm}
\noindent$\bullet$ Case B (Estimate of $\mathcal {I}_{32}$): We have
\begin{align}\label{new_12}
\int_0^{\pi}g\partial_t\log\dfrac{g}{\tilde{M}}d\theta=\int_0^{\pi}\partial_t gd\theta=0,
\end{align}
In \eqref{new_13}, we combine \eqref{new_11} and \eqref{new_12} to get the desired estimate:
\begin{equation*}
\frac{d}{dt}\mathcal{H}\left[g|\tilde{M}\right](t)=-\sigma\int_0^{\pi}r_2^2(t)\cos^4(\tilde\Phi-\theta)g(t,\theta)\left(\partial_\theta\log\dfrac{g(t,\theta)}{\tilde{M}(\theta)}\right)^2d\theta\leq 0.
\end{equation*}
\end{proof}
\section{Conclusion}
 In this paper, we have presented a generalized stochastic J-K model for nematic alignment. Based on the particle model, we derive the corresponding mean-field equation in kinetic case. For the spatially homogeneous case, we have studied the global well-posedness and regularity with constant noise. The solution is indeed nonnegative and analytic for a given nonnegative initial probability distribution with good regularity. 
 
 For the case of constant noise, all the steady states are determined explicitly. The phenomenon of phase transition is only related to the ratio between the strength of noise and alignment. When the ratio is greater than or equal to the critical value $1/4$, the equilibrium lies only in the uniform distribution. Below this threshold, a type of generalized von-Mises distribution emerges and the uniform distribution is unstable. For the subcritical case, the convergence to the uniform distribution is studied both in the $L^2$ sense and in the description of relative entropy, and there is an exponential convergence if the noise is big enough. As to the supercritical case, there is an exponential convergence to the von-Mises type equilibrium in $L^2$ sense, if the initial data is close to the equilibrium.  

  A special nonconstant noise is also considered to obtain an explicit non-symmetric steady state. 
  For the corresponding mean-field equation, there is no phase transition. We also prove that there is a non-symmetric equilibrium on the torus, and it has the same form of von-Mises distribution on both the two $\pi$-length intervals. The relative entropy to the equilibrium is non-increasing, if the order parameter $\phi_2$ is constant. 
  
  When the noise is constant and small, we only deal with the case in which the initial data is close to an equilibrium. A natural question is whether we can treat the general initial data. The convergence of the critical point case itself is also very important and interesting. Besides, for the behavior of the degenerate case, like the well-posedness and regularity, we leave them as open questions for a future work.

\section*{Acknowledgments} The work of S.-Y. Ha is supported by National Research Foundation grant funded by the Korea government (RS-2025-00514472).
The work of H. Yu is supported under Grants NSFC 12271288, and funded in part by the Austrian Science Fund (FWF) project \href{https://doi.org/10.55776/F65}{10.55776/F65}.
The work of B. Zhou is supported under Grants NSFC 12301166.

\bibliographystyle{ws-m3as} 
\bibliography{refdocument}
\end{document}